\documentclass{article}%

\usepackage{graphicx}
\usepackage{amsmath,amssymb,amsfonts}%
\usepackage{theorem}%
\usepackage{color}%
\usepackage{hyperref}
\usepackage[font={small, it}]{caption}
\theoremstyle{change}%
\newtheorem{definition}{Definition:}[section]%
\newtheorem{proposition}[definition]{Proposition:}%
\newtheorem{theorem}[definition]{Theorem:}%
\newtheorem{lemma}[definition]{Lemma:}%
{\theorembodyfont{\rmfamily}\newtheorem{remark}[definition]{Remark:}}%
{\theorembodyfont{\rmfamily}}%

\newenvironment{proof}{{\bf Proof:}}
{\qquad \hspace*{\fill} $\Box$}%

\newcommand{\fg}{\mathfrak{g}}%
\newcommand{\fn}{\mathfrak{n}}%
\newcommand{\tr}{\operatorname{tr}}%

\newcommand{\inner}{\operatorname{int}}%
\newcommand{\rme}{\mathrm{e}}%

\newcommand{\CC}{\mathcal{C}}%
\newcommand{\OC}{\mathcal{O}}%
\newcommand{\UC}{\mathcal{U}}%
\newcommand{\XC}{\mathcal{X}}%
\newcommand{\DC}{\mathcal{D}}%
\newcommand{\C}{\mathbb{C}}%
\newcommand{\N}{\mathbb{N}}%
\newcommand{\R}{\mathbb{R}}%
\newcommand{\Z}{\mathbb{Z}}%

\newcommand{\g}{\mathfrak{g}}%
\usepackage[pagewise]{lineno}

\begin{document}

\title{Control sets of one-input linear control systems on solvable, nonnilpotent 3D Lie groups}
 \author{Adriano Da Silva\thanks{Supported by Proyecto UTA Mayor Nº 4768-23} \\
		Departamento de Matem\'atica,\\Universidad de Tarapac\'a - Sede Iquique, Chile.
 		\and
 		Lino Grama and Alejandro Otero Robles\thanks{Supported by Coordena\c{c}\~ao de Aperfei\c{c}oamento de Pessoal de Nível Superior - Brasil (CAPES) - Finance Code 001} \\
 		Instituto de Matem\'{a}tica\\
		Universidade Estadual de Campinas, Brazil\\
 }
 \date{\today }

 \maketitle

 \begin{abstract}
  In this article, we completely describe the control sets of one-input linear control systems on solvable, nonnilpotent 3D Lie groups. We show that, if the restriction of the associate derivation to the nilradical is nontrivial, the Lie algebra rank condition is enough to assure the existence of a control set with a nonempty interior. Moreover, such a control set is unique and, up to conjugations, given as a cylinder of the state space. On the other hand, if such a restriction is trivial, one can obtain an infinite number of control sets with empty interiors or even controllability, depending on the group considered.
 \end{abstract}

	

\section{Introduction}

Control systems have been studied for a long time; in particular, linear control systems over $\mathbb{R}^n$ have many physical applications (see for instance \cite{Jurd, Lg, PBGM, Sk}). Roughly speaking, a control system is a system characterized by a base space (differentiable manifold) and a dynamics characterized by a family of differential equations parameterized by functions known as controls.

The first generalization of a linear control system was carried out by L. Markus in \cite{Mar} to matrices groups. Subsequently, V. Ayala and J. Tirao in \cite{AyTi} introduced the concept of linear control systems over arbitrary Lie groups. One of the main reasons to study control systems over Lie groups was provided by P. Jouan in \cite{Jouan2}, where it is shown that every control-affine system with complete vector fields that generate a finite-dimensional Lie algebra is equivalent to a linear control system over a Lie group or a homogeneous space.

In order to understand the dynamics of a control system, understanding the control sets is of great help. Control sets are the maximal regions in the base space where approximate controllability holds; they contain fixed and recurrent points and periodic and bounded orbits. Moreover, controllability holds in their interior, i.e., the possibility to steer two given points to each other through a solution of the system in positive time. For linear control system
on Lie groups, the topological properties of the control sets are connected with the eigenvalues of a derivation associated with the drift of the system. Precisely, if the group is solvable and the linear control system admits a control set containing the identity element of the group in its interior, then such a control set is unique; it is bounded if and only if the central subgroup (see Section 2.2) is compact and is open (resp. closed) if the eigenvalues of the associated derivation have only positive (resp. negative) real parts (see \cite{DSAy2, ADZ, DS}). 

Though nice, the previous results  depend on the assumption that the identity element is in the interior of the control set. However, as a consequence of \cite[Theorem 4.1]{DSAyAOR}, a linear control system can admit a control set with a nonempty interior, satisfying all the previous properties, and having the identity element in its boundary is possible. Therefore, it is natural to ask if the previous statement holds in general. In this work, we give a first step in this direction by fully characterizing the control sets of one-input linear control systems on 3D solvable, nonnilpotent Lie groups satisfying the Lie algebra rank condition. The fact that our systems have a bounded control range shows us that the dynamics are more influenced by the group structure, differing in many aspects from the unbounded case treated in \cite{DSAy1}. Basically, if the restriction of the associated derivation to the nilradical of the group is nontrivial, the existence of a unique control set with a nonempty interior is assured, and, up to conjugations, it coincides with a cylinder. Moreover, several of its topological properties are related to the eigenvalues of this restriction. On the other hand, if the restriction is trivial, very distinct scenarios can appear for different classes of groups; in one class, controllability holds under the Lie algebra rank condition, while on the other, an infinite number of control sets with empty interiors appear.

The paper is divided as follows: In Section 2, we establish the basis of our work. Here we define the concepts of control-affine systems, positive and negative orbits, control sets, controllability, and so on. We also define linear control systems on Lie groups and homogeneous spaces, together with their main properties. We also define here the concept of nilrank in a linear control system. In sequence, we specialize our discussion on 3D solvable, nonnilpotent Lie groups, where some preliminary results about conjugation on these groups are proved. In Section 3, we analyze a linear control system with nilrank two. We start by studying an associated control-affine system on $\R^2$, that is conjugated to the projection of our initial system to a specific homogeneous space. Afterwards, we state and prove our main result, showing that this class of systems admits a unique control set with a nonempty interior, given as the preimage of its counterpart on a homogeneous space. As a consequence, the topological properties of the projected system are reflected in the control set of the original system. In particular, one sees that such properties do not depend only on the eigenvalues of the associated derivation but also on the size of the control range and on the group structure. In Section 4, systems with nilrank equal to one and zero are considered. The nilrank one case is quite well behaved, in the sense that the Lie algebra rank condition is equivalent to the ad-rank condition. Such equality allows us to prove that the system also admits a unique control set with a nonempty interior and analyze its properties. The nilrank zero case is much wilder. Linear control systems on most of the classes of the groups in question do not admit control sets with nonempty interiors; instead, an infinite number of control sets with empty interiors exist. On the other hand, for one class, we do have the controllability of a linear control system. This huge difference in behavior is mainly due to the group structure, represented in this context by a $2\times 2$ matrix related to the group product. In Section 5, we consider non-simply connected Lie groups. Here there are two possible classes: the group of rigid motions and its $n$-fold covers, and the direct product of the one-dimensional torus by the 2D solvable Lie group. The results for the first class are quite straightforward since they are basically a reflection of what happens in the simply connected covering of these groups studied in the previous sections. For the second class, only nilranks equal to one and zero are possible. While for nilrank one the results are similar to the ones obtained for the simply connected groups, for the nilrank zero case, we have again a big difference between the group and its simply connected covering, since in the former we have controllability and on the latter the system admits an infinite number of control sets with empty interiors.

\section{Preliminaries}

In this section, we introduce the basic ideas about control-affine systems, control sets, linear control systems, and solvable nonnilpotent 3D Lie groups.

\subsection{Control-affine systems and control sets}

In this section, we will define the main objects that we will work with throughout this document, namely, control-affine systems and their control sets. We start with the formal definition of control-affine systems.

\begin{definition}
    A control-affine system on a connected finite-dimensional differentiable manifold $M$, is a family of ordinary differential equations

\begin{flalign*}
	  &&\dot{x}(s) = f_0(x(s)) + \sum_{i=1}^{m}u_i(s)f_i(x(s)),  \hspace{.5cm} {\bf u} = (u_1,\ldots, u_m)\in \mathcal{U}. &&\hspace{-1cm}\left(\Sigma_{M}\right)
	  \end{flalign*}

where $f_0,\ldots,f_m$  are smooth vector fields on $M$ and $\mathcal{U}$ is the set of all piecewise constant functions ${\bf u}$ satisfying ${\bf u}(s)\in \Omega$ a.e. Here, $\Omega\subset\mathbb{R}^m$ is a compact, convex subset with $0\in \inner\Omega$. The functions ${\bf u}$ are called control functions.
    
\end{definition}

For each $x\in M$ and ${\bf u}\in\UC$, the system $\Sigma_M$ admits a unique (locally) solution $s\mapsto \phi(s,x,u)$, in the sense of Caratheod\'ory, satisfying  $\phi(0,x,u) = 0$. We assume that the solutions are defined on the whole real line, since this is true for the systems we will be considering.

The set of points {\it reachable from $x$} and the set of points {\it controllable to $x$} in time $S>0$ are defined, respectively, as
$$\mathcal{O}^{+}_{S}(x):=\{y\in M|\;\;\text{there are,}\;\;{\bf u}\in \mathcal{U}\;\;\text{with}\;\; y=\phi(S,x, {\bf u})\},$$
$$\mathcal{O}^{-}_{S}(x):=\{y\in M|\;\;\text{there are,}\;\; {\bf u}\in \mathcal{U}\;\; \text{with}\;\; x=\phi(S,y, {\bf u})\}.$$
respectively. The {\it positive and negative orbits} of $x$ are
$$\mathcal{O}^{+}(x):=\bigcup_{S>0}\mathcal{O}^{+}_{S}(x)\;\;\;\text{and}\;\;\;\mathcal{O}^{-}(x):=\bigcup_{S>0}\mathcal{O}^{-}_{S}(x),$$
 respectively. We say that $\Sigma_{M}$ satisfy the \textit{Lie algebra rank condition} (LARC) if $\mathcal{L}(x)=T_x M$ for any $x\in M$, where 
        $\mathcal{L}$ is the smallest Lie algebra containing $f_0, f_1,\ldots, f_m$. In particular, if the LARC is satisfied, the sets 
         $$\mathcal{O}^{+}_{\leq S}(x):=\bigcup_{0<s\leq S}\mathcal{O}^{+}_{s}(x)\hspace{.5cm}\mbox{ and }\hspace{.5cm}\mathcal{O}^{-}_{\leq S}(x):=\bigcup_{0<s\leq S}\mathcal{O}^{-}_{s}(x), \hspace{.5cm}S>0,$$ have a nonempty interior.

In what follows, we formally define the concept of control sets.

\begin{definition}
\label{controlset}

A subset $D$ of $M$ is called a control set of the system $\Sigma_M$ if it satisfies the following properties:

\begin{itemize}
    \item [(i)] {\it (Weak invariance)} For every $x\in M$, there exists a ${\bf u}\in\mathcal{U}$ such that $\phi(\mathbb{R}^+,x,{\bf u})\subset D$;
    \item [(ii)] {\it (Approximate controllability)}  $ D\subset\overline{\mathcal{O}^{+}(x)}$ for every $x\in D$;
    \item[(iii)] {\it (Maximality)} $D$ is maximal with respect to properties (i) and (ii).
\end{itemize}
\end{definition}

In particular, when the whole state space $M$ is a control set, $\Sigma_M$ is said to be {\it controllable}. 

\bigskip

Control sets are pairwise disjoint and contain several dynamical properties of the system, such as fixed and recurrent points and periodic and bounded orbits. Moreover, exact controllability holds in its interior, i.e., points can be steered into one another by a solution of the system in positive time.  In fact, under the LARC, it holds that 
$$x\in\inner \OC^+(x)\hspace{.5cm}\iff\hspace{.5cm}x\in\inner \OC^+(x)\cap \inner \OC^-(x)\hspace{.5cm}\iff\hspace{.5cm}x\in \inner D,$$
where $D$ is a control set of $\Sigma_M$. In this case, $D=\overline{\OC^+(x)}\cap\OC^-(x)$ (See \cite[Section 3.2]{CK}).

Suppose $N$ is another smooth manifold, and 
\begin{flalign*}
	  &&\dot{z}(s) = g_0(z(s)) + \sum_{i=1}^{m}u_i(s)g_i(z(s)),  \hspace{.5cm} {\bf u} = (u_1,\ldots, u_m)\in \mathcal{U}. &&\hspace{-1cm}\left(\Sigma_{N}\right)
	  \end{flalign*}
is an affine control system over $N$. If $\psi:M\rightarrow N$ is a surjective smooth map, we say that $\Sigma_M$ and $\Sigma_N$ are $\psi$-conjugate if their respective vector fields are $\psi$-conjugate, i.e

$$\psi_*f_i=g_i\circ\psi,\;\;\;j=0,\ldots m.$$

If such $\psi$ exists, we say that $\Sigma_M$ and $\Sigma_N$ are $\psi$-conjugate. If $\psi$ is a diffeomorphism, we say that $\Sigma_M$ and $\Sigma_N$ are equivalent. The next result relates control sets of conjugated systems. Its proof is standard and will be omitted.

\begin{proposition}
\label{conjugation}
    Let $\Sigma_M$ and $\Sigma_N$ be $\psi$-conjugated systems. It holds:
\begin{enumerate}
    \item[1.] If $D$ is a control set of $\Sigma_M$, there exists a control set $E$ of $\Sigma_N$ such that $\psi(D)\subset E$;
    \item[2.] If for some $y_0\in \inner E$ it holds that $\psi^{-1}(y_0)\subset\inner D$, then $D=\psi^{-1}(E)$.
\end{enumerate}
    
\end{proposition}

 \begin{remark}
     \label{function}

 Before finishing the section, let us make a simple remark that will help us ahead.  Let $f:M\rightarrow\R$ be a continuous function and $\Sigma_G$ a control-affine system on $M$. If $x, y\in M$ are such that,
     $$f(x)<f(\phi(s, y, {\bf u})), \hspace{.5cm}\forall s\geq 0, {\bf u}\in\UC \hspace{.5cm}\mbox{ then } \hspace{.5cm}x\notin\overline{\OC^+(y)}.$$
     In particular, it cannot exist a control set of $\Sigma_M$ that contains both $x$ and $y$, since the condition (ii) of Definition \ref{controlset} cannot be satisfied.

 \end{remark}

\subsection{Linear vector fields and linear control systems on Lie groups}

In this section, we will define linear control systems on Lie groups. We explore the properties satisfied by these systems over Lie groups, with special attention over solvable, nonnilpotent 3D Lie groups. We begin with the definition of a linear vector field over a Lie group.

\begin{definition}
Let $G$ be a connected Lie group with Lie algebra $\mathfrak{g}$ identified with the set of left-invariant vector fields.  A vector field $\mathcal{X}$ on $G$ is called linear if its associated flow $\{\varphi_s\}_{s\in\R}$,  is a one-parameter subgroup of automorphisms of $G$, that is,
        $$\forall s\in\R, \hspace{1cm}\varphi_s(gh)=\varphi_s(g)\varphi_s(h).$$
       
\end{definition}

The vector field $\mathcal{X}$ is complete and is uniquely associated with a derivation $\DC$ of $\g$ defined by

$$\mathcal{D}Y=-[\mathcal{X},Y](e), \;\;\;\forall Y \in \g.$$

Moreover, since  $\varphi_s\in\mathrm{Aut}(G)$ implies that $(d\varphi_s)_{e}\in\mathrm{Aut}(\fg)$ for all $s\in \R$, it holds that

\begin{equation}
    \label{formula}
    (d\varphi_s)_e=\mathrm{e}^{s\mathcal{D}},\;\;\;\forall s \in \mathbb{R}.
\end{equation}

This nice relationship between the flow of $\XC$ and the derivation $\DC$ has several important applications, mainly due to the subgroups and subalgebras associated with them, as we define next.

Let us consider an eigenvalue $\alpha\in\C$ of the derivation $\DC$. The real generalized eigenspaces of $\DC$ are the subspaces of $\fg$ defined as 
$$
\mathfrak{g}_{\alpha}=\{X\in \mathfrak{g}:(\mathcal{D}-\alpha I)^{n}X=0\;\;\mbox{for some }n\geq 1\}, \;\;\mbox{ if }\;\;\alpha\in\R \;\;\mbox{ and}
$$
$$\mathfrak{g}_{\alpha}=\mathrm{span}\{\mathrm{Re}(v), \mathrm{Im}(v);\;\;v\in \bar{\fg}_{\alpha}\},\;\;\mbox{ if }\;\;\alpha\in\C,$$
where $\bar{\fg}=\fg+i\fg$ stands for the complexification of $\fg$, and $\bar{\fg}_{\alpha}$ is the generalized eigenspace of the extension $\bar{\DC}=\DC+i\DC$ of $\DC$ to $\bar{\fg}$.  By Proposition 3.1 of \cite{SM1}, it holds that
$$[\bar{\fg}_{\alpha },\bar{\fg}_{\beta }]\subset 
\bar{\fg}_{\alpha +\beta },$$
where $\bar{\fg}_{\alpha +\beta }=\{0\}$ if $\alpha +\beta $ is not an eigenvalue of $\mathcal{D}$. Therefore, if we put
$$\fg_{\lambda}:=\bigoplus_{\alpha; \mathrm{Re}(\alpha)=\lambda}\fg_{\alpha},$$
the previous imply that 
$$[\fg_{\lambda_1}, \fg_{\lambda_2}]\subset \fg_{\lambda_1+\lambda_2}\;\;\;\mbox{ when }\lambda_1+\lambda_2=\mathrm{Re}(\alpha)\;\mbox{ for some eigenvalue }\alpha\;\mbox{ of }\;\DC\;\mbox{ and zero otherwise},$$
where $\fg_{\lambda}=\{0\}$ if $\lambda\in\R$ is not the real part of any eigenvalue of $\DC$. 

The {\it unstable, central,} and {\it stable} subalgebras of $\fg$ are given, respectively, by
\begin{equation*}
\mathfrak{g}^{+}=\bigoplus_{\alpha :\, \mathrm{Re}(\alpha)>	0}\mathfrak{g}_{\alpha },\hspace{1cm}\mathfrak{g}^{0}=\bigoplus_{\alpha :\,%
	\mathrm{Re}(\alpha )=0}\mathfrak{g}_{\alpha }\hspace{1cm}%
\mbox{ and }\hspace{1cm}\mathfrak{g}^{-}=\bigoplus_{\alpha :\, \mathrm{Re}%
	(\alpha )<0}\mathfrak{g}_{\alpha }.
\end{equation*}

Since the previous subalgebras are given as sum of all the (generalized) eigenspaces of $\DC$, it holds that  $\mathfrak{g}=\mathfrak{g}^{+}\oplus \mathfrak{g}^{0}\oplus \mathfrak{g}^{-}$. Moreover, these subalgebras are invariant by the derivation $\DC$, with $\mathfrak{g}^{+}$ and $\mathfrak{g}^{-}$ nilpotent subalgebras.

The relationship between the previous subalgebras and the dynamics of the flow of $\XC$ is obtained through their subgroups. Let us denote by $G^{+}$, $G^{-}$, $G^{0}$, $G^{+,0},$ and $G^{-,0}$, the connected Lie subgroups whose associated Lie algebras are given, respectively, by $\mathfrak{g}^{+}$, $%
\mathfrak{g}^{-}$, $\mathfrak{g}^{0}$, $\mathfrak{g}^{+,0}:=\mathfrak{g}%
^{+}\oplus \mathfrak{g}^{0}$ and $\mathfrak{g}^{-,0}:=\mathfrak{g}^{-}\oplus 
\mathfrak{g}^{0}$. By Proposition 2.9 of \cite{DS}, all of the previous subgroups are invariant by the flow of $\XC$, closed, and their intersections are trivial, that is, 
$$G^+\cap G^-=G^+\cap G^{-, 0}=\ldots=\{e\}.$$
Moreover, $G^+$ and $G^-$ are connected, simply connected, nilpotent Lie groups. Analogously, at the algebra level, the subgroups $G^+, G^0$, and $G^-$ are called, respectively, the {\it unstable, central,} and {\it stable subgroups} of $\XC$.

Next, we define linear control systems. A {\it linear control system (LCS for short)} on $G$ is determined by the family of ODEs
\begin{equation}\label{scl}
\tag{$\Sigma_G$}
    \dot{g}(s) = \mathcal{X}(g(s)) + \sum_{i=1}^{m}u_i(s)Y^{i}(g(s)),
\end{equation}

where $\mathcal{X}$ is a linear vector field, $Y^{i}$ are left-invariant vector fields, and ${\bf u} = (u_1,\ldots, u_m)\in \mathcal{U}$ are control functions as defined previously.

The solutions of $\Sigma_G$ satisfy 
$$\phi(t, gh, {\bf u})=\varphi_t(g)\phi(t, h, {\bf u}), \hspace{1cm}\forall g, h\in G, t\in\R, {\bf u}\in\UC.$$
A consequence of the previous formula is the following proposition (see \cite[Proposition 2]{Jouan})

\begin{proposition}
\label{properties}
    For a LCS, it holds that
    \begin{enumerate}
    \item $\OC^{+}_{S_1+S_2}(e)=\varphi_{S_2}(\OC^+_{S_1}(e))\OC_{S_2}^+(e)$, for all $S_1, S_2>0$;
    \item $\OC^{+}_{S_1}(e)\subset \OC^{+}_{S_2}(e)$, for all $0<S_1<S_2$;
    \item $\OC^{+}_{\leq S}(e)=\OC^{+}_{S}(e),$ for all $S>0$.
\end{enumerate}
\end{proposition}

Also, due to the symmetry present in Lie groups, the LARC for linear control systems is determined at the origin of the group; in fact, the linear control system $\Sigma_G$ satisfies the LARC if $\g$ is the smallest $\DC$-invariant subalgebra containing the vectors $\{Y^1, \ldots, Y^m\}$.

A strong algebraic condition is given by the \textit{ad-rank condition}. We say that $\Sigma_G$ satisfies the ad-rank condition if $\g$ is the smallest $\mathcal{D}$-invariant subspace containing the vectors $\{Y^1, \ldots, Y^m\}$. The difference between the two conditions is given by the brackets among the elements in $\{Y^1, \ldots, Y^m\}$ that are present in the LARC and not on the ad-rank. In particular, if the system satisfies the ad-rank condition, it is locally controllable; that is,
   $$e\in\inner\OC^+_{S}(e)\cap \inner\OC^-_{S}(e),\hspace{.5cm}\forall S>0.$$

As showed in \cite{ADZ, DS}, there is a nice relationship between the subgroups $G^+, G^0$ and $G^-$ associated with the drift $\XC$ of a LCS $\Sigma_G$ and its dynamics. In fact, the following result holds:

\begin{theorem}
\label{subgrupos}
    Let $\Sigma_G$ be a LCS on a connected Lie group $G$ and assume that $\OC^+(e)$ is a neighborhood of the identity element $e\in G$. Then,
$$G^{-, 0}\subset\OC^+(e)\hspace{.5cm} \mbox{ and } \hspace{.5cm}G^{+, 0}\subset\OC^-(e).$$

Moreover, 
$$\CC_G=\overline{\OC^+(e)}\cap\OC^-(e),$$
is a control set of $\Sigma_G$ with a nonempty interior satisfying $G^0\subset \inner\CC_G$. If $G$ is solvable, $\CC_G$ is the unique control set of $\Sigma_G$ with nonempty interior, and $\Sigma_G$ is controllable if $G=G^0$ or, equivalently, if the derivation $\DC$ has only eigenvalues with zero real parts.

\end{theorem}

We finish this section with a definition that will be useful to divide our analyses into cases.

\begin{definition}
    \label{nilrank}

    Let $\Sigma_{G}$ be a LCS on a connected Lie group $G$ with Lie algebra $\fg$. We define the {\it nilrank} of $\Sigma_G$ as the rank of the restriction $\DC|_{\fn}$, where $\DC$ is the derivation associated with the drift of $\Sigma_G$ and $\fn$ is the nilradical of $\fg$.
\end{definition}

\subsection{Lifting and projecting LCSs}

Let $G$ be a connected Lie group with Lie algebra $\fg$ and denote by $\widetilde{G}$ its simply connected covering group, chosen in such a way that the canonical projection $\pi:\widetilde{G}\rightarrow G$ satisfies $(d\pi)_{\widetilde{e}}=I_{\fg}$, and hence $\pi\circ\widetilde{\exp}=\exp$, where $\widetilde{\exp}$ and $\exp$ stand, respectively, for the exponential maps of $\widetilde{G}$ and $G$.

Now, if $\XC$ be a linear vector field on $G$ and denote by $\DC$ its associated derivation. By the previous choice and Theorem 2.2 of \cite{AyTi}, the derivation $\DC$ is associated with a unique linear vector field $\widetilde{\XC}$ on $\widetilde{G}$. If we denote by $\{\widetilde{\varphi}_s\}_{s\in\R}$ and $\{\varphi_s\}_{s\in\R}$, respectively, the flows of $\widetilde{\XC}$ and $\XC$, then
$$\forall X\in\fg, s\in\R \hspace{.5cm}\pi\left(\widetilde{\varphi}_s(\widetilde{\exp} X)\right)=\pi\left(\widetilde{\exp}\left(\rme^{s\DC}X\right)\right)=\exp \left(\rme^{s\DC}X\right)=\varphi_s(\exp X)=\varphi_s(\pi(\widetilde{\exp} X)).$$
By the connectedness of $\widetilde{G}$, we conclude that 
$$\forall s\in\R, \hspace{1cm}\pi\circ\widetilde{\varphi}_s=\varphi_s\circ\pi,$$
and hence $\widetilde{\XC}$ and $\XC$ are $\pi$-related vector fields.

In the same way, the property $\pi\circ\widetilde{\exp}=\exp$ implies that 
$$\forall s\in\R, Y\in\fg, g\in\widetilde{G}, \hspace{1cm} \pi(g\,\widetilde{\exp} \, sY)=\pi(g)\exp sY,$$
which implies that the left-invariant vector fields $\widetilde{Y}$ on $\widetilde{G}$ and $Y$ on $G$ determined by $Y$ are $\pi$-conjugated. As a consequence, any LCS $\Sigma_G$ on $G$ can be lifted (uniquely) to a LCS $\widetilde{\Sigma}_{\widetilde{G}}$ on $\widetilde{G}$ in such a way that $\widetilde{\Sigma}_{\widetilde{G}}$ and $\Sigma_G$ are $\pi$-conjugated.

\bigskip

Let $H\subset G$ be a closed subgroup. Following \cite{Jouan2}, a vector field $f$ on the homogeneous space $H\setminus G$ is said to be linear if it is $\pi$-conjugated to a linear vector field $\XC$ on $G$, where $\pi:G\rightarrow H\setminus G$ is the canonical projection, that is, 
$$f\circ\pi=\pi_*\circ \XC.$$
Moreover, by Proposition 4 of \cite{Jouan2} the previous is equivalent to the invariance of the subgroup $H$ by the flow $\{\varphi_t\}_{t\in\R}$ of $\XC$. In fact, if
$$\forall s\in\R\hspace{1cm}\varphi_s(H)\subset H,$$
the relation,
$$\Phi_s(\pi(g))=\pi(\varphi_s(g))\hspace{1cm}\forall g\in G, s\in\R,$$
 defines a flow $\Phi:\R\times H\setminus G\rightarrow H\setminus G$ on the homogeneous space $H\setminus G$, whose associated vector field 
 $$f(x):=\frac{d}{ds}_{|s=0}\Phi_s(x),$$
is $\pi$-conjugated to $\XC$.

As a consequence, any LCS $\Sigma_G$ is $\pi$-conjugated to a control-affine system $\Sigma_{H\setminus G}$ (also called a linear control system) if the subgroup $H$ is invariant by the flow of the drift $\XC$ of $\Sigma_G$.

\subsection{3D solvable, nonnilpotent Lie groups} \label{section3D}

We will consider here only the three-dimensional connected and simply connected Lie groups that are solvable and nonnilpotent, since from the general theory of Lie groups, the connected ones can be obtained as homogeneous spaces. 

By \cite[Theorem 1.4, Chapter 7]{onis}, the real Lie algebras of dimension 3 are given as a semidirect product of $\mathbb{R}$ with $\mathbb{R}^2$, i.e., $\g(\theta):= \mathbb{R}\times_{\theta} \mathbb{R}^2$ where $\theta(t):=t\theta$, with $\theta$ a $2\times 2$ matrix, which can take one of the following forms:

\begin{itemize}
    \item $\theta = \left(\begin{array}{cc}
                                 1 & 1 \\
                                 0 & 1
                          \end{array}\right)$
    \item $\theta = \left(\begin{array}{cc}
                                 1 & 0 \\
                                 0 & \gamma
                          \end{array}\right)$ with $|\gamma|\leq 1$
    \item $\theta = \left(\begin{array}{cc}
                                 \gamma & -1 \\
                                 1 & \gamma
                               \end{array}\right)$ with $\gamma \in \mathbb{R}$                      
\end{itemize}
% \begin{center}
% $\left(\begin{array}{cc}
%     1 & 1 \\
%     0 & 1
% \end{array}\right),\;\;\;\;
% \left(\begin{array}{cc}
%     1 & 0 \\
%     0 & \gamma
% \end{array}\right)$ com $|\gamma|\leq 1$ e 
% $\left(\begin{array}{cc}
%     \gamma & -1 \\
%     1 & \gamma
% \end{array}\right)$ com $\gamma \in \mathbb{R}$
% \end{center}
The bracket of these Lie algebras is completely determined by the equation
$$[(t,0),(0,v)]=(0,t\theta v), \hspace{1cm}t\in\R, v\in\R^2.$$
Up to isomorphism, the unique connected, simply connected Lie group associated with the Lie algebra $\fg(\theta)$, is given by the semi direct product $\mathbb{R}\times_{\rho}\mathbb{R}^2$, with 
$$\rho:\R\rightarrow \mathrm{Gl}(2, \R), \hspace{1cm}\rho_t:=\mathrm{e}^{t\theta},$$
and the product of the group as
$$(t_1,v_1)(t_2,v_2)=(t_1+t_2,v_1 + \rho_{t_1}v_2)$$

with $(t_1,v_1), (t_2,v_2)\in \mathbb{R}\times_{\rho}\mathbb{R}^2$. Due to the dependence on $\theta$, we use the notation $G(\theta)=\R\times_{\rho}\R^2$, if $\rho_1=\rme^{\theta}$. 

For all $(t, v)\in G(\theta)$ and $w\in\R^2$, we have that
$$(t, v)(0, w)(t, v)^{-1}=(t, v+\rho_t w)(-t, -\rho_{-t}v)=(0, v+\rho_t w+\rho_t(-\rho_{-t}v)=(0, \rho_tw).$$
Let us use the previous to calculate some homogeneous spaces of $G(\theta)$. Precisely, let $w\in\R^2$ be a nonzero vector, and consider the subgroup $H=\R(0, w)$. By the previous, we get that 
$$H\cdot(t_1, v_1)=H \cdot(t_2, v_2)\hspace{.5cm}\iff\hspace{.5cm} t_1=t_2\hspace{.5cm}\mbox{ and } \hspace{.5cm} \langle v_1, R w\rangle=\langle v_2, R w\rangle,$$
implying that the canonical projection is given by $$\pi: G(\theta)\rightarrow H\setminus G(\theta), \hspace{1cm} \pi(t, v)=(t, \langle v, R w\rangle),$$
that is, $\pi$ coincides with the projection onto the first two components of the basis $\{(1, 0), (0, Rw), (0, w))\}$, when $G(\theta)$ is seen as the vector space $\R^3$. Moreover, $H\setminus G(\theta)$ is also a Lie group if and only if $H$ is a normal subgroup of $G(\theta)$ if and only if $w$ is an eigenvector of $\theta$.

On the other hand, the subgroup $S=\R(1, 0)$ is such that 
$$(t, 0)(t_1, v_1)=(t_2, v_2)\hspace{.5cm}\iff\hspace{.5cm} t=t_2-t_1\hspace{.5cm}\mbox{ and } \hspace{.5cm} \rho_tv_1=v_2,$$
and hence 
$$S\cdot(t_1, v_1)=S\cdot(t_2, v_2)\hspace{.5cm}\iff\hspace{.5cm} \rho_{-t_1}v_1=\rho_{-t_2}v_2.$$
As a consequence, the homogeneous space $S\setminus G(\theta)$ is naturally identified with $\R^2$ and its canonical projection is given by $$\pi:G(\theta)\rightarrow S\setminus G(\theta), \hspace{.5cm} \pi(t, v)=\rho_{-t}v.$$
In what follows, we define linear and left-invariant vector fields. Following \cite{DSAy1}, the left-invariant and linear vector fields on the groups $G(\theta)$ are given, respectively, as 

\begin{equation}\label{cii}
    Y^L(t,v)=(\alpha,\rho_t\eta)\hspace{.5cm}\mbox{ and }\hspace{.5cm}\mathcal{X}(t,v)=(0,Av + \Lambda_t^{\theta}\xi),
\end{equation}

where $(\alpha,\eta)\in \mathbb{R}\times\mathbb{R}^2$,  $\xi\in \mathbb{R}^2$ and $A\in \mathfrak{gl}(2,\mathbb{R})$ satisfies $A\theta=\theta A$. Moreover, for any $2\times 2$ matrix $B$, we have defined
$$
    \Lambda^{B}:\mathbb{R}\times \mathbb{R}^2\longrightarrow \mathbb{R}^{2},\;\;\; (t,v)\mapsto \int_0^t\mathrm{e}^{sB}vds.
$$
Such operator appears not only in the expression of linear vector fields but also in the expressions of the exponential map and the automorphisms of $G(\theta)$ (See \cite{DSAy1}).

Since the linear vector field $\mathcal{X}$ is fully characterized by the matrix $A$ and the vector $\xi$, we will usually write $\mathcal{X}=(A, \xi)$ to represent the linear vector field. The same holds for a left-invariant vector field, which we usually denote by $Y=(\alpha, \eta)$.

The flow associated with 
 the linear vector field $\mathcal{X}=(A, \xi)$ and the associated derivation are given explicitly by  

\begin{equation}
    \label{linearvector}
    \varphi_s(t,v)=(t,\mathrm{e}^{sA}v + \Lambda_t^\theta\Lambda_s^A\xi)\hspace{.5cm}\mbox{ and }\hspace{.5cm} \DC=\left(\begin{array}{cc}
        0 & 0 \\
        \xi &  A
    \end{array}\right).
\end{equation}

\subsection{Linear control systems on $G(\theta)$}

In this section, we define the linear control systems for the groups $G(\theta)$ we will be studying and comment on some of their properties.

\begin{definition}\label{system3d}
A (one-input) linear control system on $G(\theta)$ is given, in coordinates, by the family of ODE's 
 \begin{flalign*}
	  &&\left\{\begin{array}{l}
     \dot{t}=u\alpha\\
     \dot{v}=Av+\Lambda_t^{\theta}\xi+u\rho_t\eta
\end{array}\right.  &&\hspace{-1cm}\left(\Sigma_{G(\theta)}\right)
	  \end{flalign*}

with $\Omega=[u^-, u^+]$ and $u^-<0<u^+$. 
\end{definition}

It is straightforward to see that $\Sigma_{G(\theta)}$ satisfies the LARC if and only if 
\begin{equation}
\label{LARC}
    \alpha\left(\bigl\langle A(\alpha\xi+A\eta), R(\alpha\xi+A\eta)\bigr\rangle^2+\bigl\langle \theta(\alpha\xi+A\eta), R(\alpha\xi+A\eta)\bigr\rangle^2\right)\neq 0,
\end{equation}

and the ad-rank condition, if and only if
\begin{equation}
\label{adrank}
    \alpha\bigl\langle A(\alpha\xi+A\eta), R(\alpha\xi+A\eta)\bigr\rangle\neq 0,
\end{equation}
where $\langle\cdot, \cdot\rangle$ stands for the canonical product in $\R^2$ and $R$ by the counter-clockwise rotation of $\frac{\pi}{2}$-degrees. Geometrically, the previous states that $\Sigma_{G(\theta)}$ satisfies the LARC if and only if $\alpha\neq 0$ and $\alpha\xi+A\eta$ is not a common eigenvector of $A$ and $\theta$. Analogously, it satisfies the ad-rank condition if and only if $\alpha\neq 0$ and $\alpha\xi+A\eta$ is not an eigenvector of $A$.

In what follows, we state and prove some results on the conjugation of the systems $\Sigma_{G(\theta)}$. The aim is to find equivalent systems whose analysis is somehow simpler.

\begin{proposition}
\label{conjugation2}
Let $\Sigma_{G(\theta)}$ be a LCS on $G(\theta)$ with associated linear vector field $\XC=(A, \xi)$ and left-invariant vector field $Y=(\alpha, 0)$. Assume that $\Sigma_{G(\theta)}$ satisfies the LARC. It holds:

\begin{enumerate}
    \item There exists $\tilde{\psi}\in\mathrm{Aut}(G(\theta))$ conjugating $\Sigma_{G(\theta)}$ to a LCS $\tilde{\Sigma}_{G(\theta)}$ with left-invariant vector field $\tilde{Y}=(\alpha, \eta)$;

    \item If $\det A\neq 0$, there exists $\hat{\psi}\in\mathrm{Aut}(G(\theta))$ conjugating $\Sigma_{G(\theta)}$ to a LCS $\hat{\Sigma}_{G(\theta)}$ with linear vector field $\hat{\XC}=(A, 0)$.
\end{enumerate}

\end{proposition}

\begin{proof}
Following \cite[Proposition 2.1 and 2.2]{DSAyDG}, for any $\delta\in\R^2$ and $P\in \mathrm{Gl}(2, \R)$, with $P\theta=\theta P$, the map
$$\psi:G(\theta)\rightarrow G(\theta), \hspace{1cm}\psi(t, v)=(t, Pv+\Lambda_t^{\theta}\delta),$$
is an automorphism of $G(\theta)$ whose differential is given by 
$$(d\psi)_{(t, v)}(a, w)=(a, Pw+a\rho_t\delta).$$

1. Since we are assuming the LARC, $\alpha\neq 0$, and hence, 
    \begin{equation*}
        \tilde{\psi}:G(\theta)\rightarrow G(\theta), \hspace{1cm}\tilde{\psi}(t, v)=(t, v-\alpha^{-1}\Lambda^{\theta}_t\eta),
    \end{equation*}
   is an automorphism, satisfying
$$(d\tilde{\psi})_{(t, v)}Y^L(t, v)= (d\tilde{\psi})_{(t, v)}(\alpha, \rho_t\eta) =(\alpha, \rho_t\eta-\alpha\alpha^{-1}\rho_t\eta)=(\alpha, 0)=\tilde{Y}(\psi(t, v)),$$
and, if $\tilde{\XC}=(A, \xi+\alpha^{-1}A\eta)$, then
$$\hspace{-1cm}(d\psi)_{(t, v)}\XC(t, v)=(d\psi)_{(t, v)}(0, Av+\Lambda^{\theta}_t\xi)=(0, Av+\Lambda^{\theta}_t\xi)$$
$$\hspace{2cm}=\bigl(0, A(v- \alpha^{-1}\Lambda_t^{\theta}\eta)+\Lambda_t^{\theta}(\xi+\alpha^{-1}A\eta)\bigr)=\tilde{\XC}(\tilde{\psi}(t, v)),$$
showing that $\Sigma_{G(\theta)}$ is $\tilde{\psi}$-conjugated to the system $\tilde{\Sigma}_{G(\theta)}$ determined by $\tilde{\XC}$ and $\tilde{Y}$.

2. On the other hand, if $\det A\neq 0$ the map 
$$\hat{\psi}:G(\theta)\rightarrow G(\theta), \hspace{1cm} \hat{\psi}(t, v)=(t, v+\Lambda_tA^{-1}\xi),$$
is a well defined automorphism of $G(\theta)$ 
 satisfying
$$\forall (a, w)\in T_{(t, v)}G(\theta), \hspace{1cm}(d\hat{\psi})_{(t, v)}(a, w)=(a, w+a\rho_tA^{-1}\xi).$$
Moreover, 
$$A\theta=\theta A\;\;\;\;\mbox{ implies that }\;\;\;\;\forall t\in\R, w\in\R^2, \;\;\; A\Lambda^{\theta}_t w=\Lambda^{\theta}_t Aw,$$
and hence,  
$$(d\hat{\psi})_{(t, v)}\XC(t, v)=\hat{\XC}(\hat{\psi}(t, v))\hspace{.5cm}\mbox{ and }\hspace{.5cm}(d\hat{\psi})_{(t, v)}Y(t, v)=\hat{Y}(\hat{\psi}(t, v)),$$
where $\hat{\XC}(t, v)=(0, Av)$ and $\hat{Y}=(\alpha, \alpha A^{-1}\xi+\eta)$, showing that $\Sigma_{G(\theta)}$ is equivalent to the LCS $\hat{\Sigma}_{G(\theta)}$ determined by $\hat{\XC}$ and $\hat{Y}$.    
\end{proof}

\bigskip

For a LCS $\Sigma_{G(\theta)}$ on $G(\theta)$ with associated linear vector field $\XC=(A, \xi)$, the definition of nilrank (Definition \ref{nilrank}) coincides with the rank of the matrix $A$. In fact, since the derivation $\DC$ associated to $\XC=(A, \xi)$ is $\DC=\left(\begin{array}{cc}
   0  & 0 \\
    \xi & A
\end{array}\right)$ and the nilradical of $\fg(\theta)$ is $\fn=\{0\}\times\R^2$, we have that 
$\DC|_{\fn}=A$.

The next result shows that, up to equivalence, the dynamics of a LCS on $G(\theta)$, with nilrank two, is the same as the dynamics of the product of a homogeneous system on $\R$ and the linear control system induced on the homogeneous space $G^0\setminus G(\theta)$, where $G^0$ is the set of singularites of the drift.

\begin{proposition}
\label{conjugation3}
Let $\Sigma_{G(\theta)}$ be a LCS with nilrank two on $G(\theta)$. Then, there exist a diffeomorphims $\psi:G(\theta)\rightarrow \R\times\R^2$ that conjugates $\Sigma_{G(\theta)}$ to the control-affine system on $\R\times\R^2$ 
 \begin{flalign*}
	  &&\left\{\begin{array}{l}
     \dot{t}=u\alpha\\
     \dot{v}=(A-u\alpha\theta)v+u\eta
\end{array}\right.  &&\hspace{-1cm}\left(\Sigma_{\R\times\R^2}\right)
	  \end{flalign*}

Moreover, if $G^0$ is the set of singularities of $\XC$, the linear control system $\Sigma_{G^0\setminus G(\theta)}$ induced on the homogeneous space $G^0\setminus G(\theta)$ is equivalent to 
 \begin{flalign*}
	  &&
     \dot{v}=(A-u\alpha\theta)v+u\eta, \; u\in\Omega
  &&\hspace{-1cm}\left(\Sigma_{\R^2}\right)
	  \end{flalign*}
   by the unique diffeomorphism $$\widetilde{\psi}: G^0\setminus G(\theta)\rightarrow \R^2\hspace{.5cm} \mbox{ satisfying }\hspace{.5cm}\widetilde{\psi}\circ\pi=\pi_2\circ\psi,$$
   where $\pi:G(\theta)\rightarrow G^0\setminus G(\theta)$ is the canonical projection and $\pi_2:\R\times\R^2\rightarrow\R^2$ the projection onto the second component.
\end{proposition}

\begin{proof} Since $\Sigma_{G(\theta)}$ has nilrank two, we have that $\det A\neq 0$ and by the previous proposition, we can assume, w.l.o.g., that $\Sigma_{G(\theta)}$ is determined by the vectors 
$\XC=(A, 0)$ and $Y=(\alpha, \eta)$.

Define the map 
$$\psi:G(\theta)\rightarrow \R\times\R^2, \hspace{1cm}\psi(t, v)=(t, \rho_{-t}v).$$
It holds that $\psi$ is a diffeomorphism, and it satisfies 
$$\forall (a, w)\in T_{(t, v)}G(\theta), \hspace{1cm}(d\psi)_{(t, v)}(a, w)=(a, -a\theta\rho_{-t}v+\rho_{-t}w).$$
Consequently,
$$(d\psi)_{(t, v)}\XC(t, v)=\XC(\psi(t, v))\hspace{.5cm}\mbox{ and }\hspace{.5cm} (d\psi)_{(t, v)}Y^L(t, v)=Z(\psi(t, v)),$$
where $Z(t, v)=(\alpha, -\alpha \theta v+\eta)$.

 Therefore, $\Sigma_{G(\theta)}$ is equivalent to the control-affine system on $\R\times\R^2$ given by
 \begin{flalign*}
	  &&\left\{\begin{array}{l}
     \dot{t}=u\alpha \\
     \dot{v}=(A-u\alpha\theta)v+u\eta
\end{array}\right., \; u\in\Omega &&\hspace{-1cm}\left(\Sigma_{\R\times\R^2}\right)
	  \end{flalign*}
showing the first assertion. On the other hand, if $\pi_2:\R\times\R^2\rightarrow \R^2$ is the projection onto the second component, 
$$\pi_2\circ\psi(t, v)=\pi_2(t, \rho_{-t}v)=\rho_{-t}v=\pi(t, v),$$
where $\pi: G(\theta)\rightarrow S\setminus G(\theta)$ is the canonical projection as shown in the beginning of Section \ref{section3D}. Since $S=\R(1, 0)=G^0$ is exactly the set of singularities of $\XC=(0, A)$, we conclude that the linear control system $\Sigma_{G^0\setminus G(\theta)}$ induced by $\pi$ coincides with the projection onto $\R^2$ of the system $\Sigma_{\R\times\R^2}$ obtained previously, that is, $\widetilde{\psi}=I_{\R^2}$ conjugates $\Sigma_{G^0\setminus G(\theta)}$ to the control-affine system
\begin{flalign*}
	  &&
     \dot{v}=(A-u\alpha\theta)v+u\eta, \; u\in\Omega,
  &&\hspace{-1cm}\left(\Sigma_{\R^2}\right)
	  \end{flalign*}
   concluding the proof.
\end{proof}

\begin{remark}
Although the subjacent manifold of $G(\theta)$ is $\R\times\R^2$, the change in notations made in the previous result is to emphasize the fact that the system $\Sigma_{\R\times\R^2}$ is not a linear control system. When working with such a system, we will always use the previous change in notation.
\end{remark}

\section{LCSs with nilrank two on $G(\theta)$}

In this section, we study the LCSs of the groups $G(\theta)$ having nilrank-two. In order to do that, we start with a full investigation of a particular class of control-affine systems on $\R^2$, whose dynamics is intrinsically associated with the LCSs on $G(\theta)$.

\subsection{Control-affine systems on $\R^2$}

In this section, we study the class of control-affine systems over $\R^2$ given by the family of differential equations.
\begin{flalign*}
	  &&
     \dot{v}=(A-u\theta)v+u\eta, \; u\in\Omega,
  &&\hspace{-1cm}\left(\Sigma_{\R^2}\right)
	  \end{flalign*}

where $\eta\in\R^2$ is a fixed nonzero vector and $A, \theta\in\mathrm{gl}(2, \R)$ satisfying $\det A\neq 0$ and $[A,\theta]=0$ \footnote{ The conditions on the matrices $A$ and $\theta$ are in accordance with the ones satisfied by linear control systems on the homogeneous spaces $S\setminus G(\theta)$ as obtained in Section 2.}. Moreover, the LARC for $\Sigma_{\R^2}$ is given by equation (\ref{LARC}) with $\alpha=1$.

\begin{remark}
It is important to comment that a more general study of control-affine systems on higher-dimensional Euclidean spaces was done recently in \cite{CK2}. Despite this fact, a full independent analysis of the system $\Sigma_{S\setminus G(\theta)}$ is done here by considering the dynamics of $2\times 2$ matrices.
\end{remark}

\bigskip

Let us define $A(u):=A-u\theta$. The solutions of $\Sigma_{\R^2}$ are built through concatenations of the solutions for constant controls $u\in\Omega$ given by

$$\phi(s, v, u)=\rme^{sA(u)}v_0+\int_0^s\rme^{(s-\tau)A(u)}u\eta d\tau.$$
In particular, if $\det A(u)\neq 0$, we get that 
$$\phi(s, v, u)=\rme^{sA(u)}(v_0-v(u))+v(u),\hspace{.5cm}\mbox{ where }\hspace{.5cm} v(u):=-uA(u)^{-1}\eta,$$
are equilibrium points of the $\Sigma_{S\setminus G(\theta)}$, that is, they satisfy the equation
\begin{equation}\label{equilibrium}
    A(u)v+u\eta=0.
\end{equation}

Moreover, by a simple calculation, one can show that  
\begin{equation}
    \label{formula}
\phi(s,v,\textbf{u})=\rme^{\sum_{i=1}^{n}s_iA(u_i)}v+\phi(s,0,\textbf{u}),
\end{equation}
where $0=s_0<s_1<\cdots<s_n$ and ${\bf u}_{|[s_j, s_{j+1})}=u_j\in\Omega$ for $j=0, \ldots, n-1$. Such property will come into play when analyzing the periodic orbits ahead.

Since we are assuming that $\det A\neq0$, the subset of $\Omega$ given by 
$$\widehat{\Omega}:=\{u\in\inner\Omega; \;\det A(u)\neq 0\},$$
is an open neighborhood of $0\in\R$. Moreover, the map 
$v: u\in\widehat{\Omega}\mapsto  v(u)=-uA(u)^{-1}\eta\in \R^2,$ 
is a smooth, regular curve. In fact, a simple derivation of the equation \ref{equilibrium}, gives us that 

\begin{equation*}
    -\theta v(u)+A(u)v'(u)=-\eta\hspace{.5cm}\implies\hspace{.5cm}v'(u)=-A(u)^{-1}(\eta-\theta v(u)).
\end{equation*}

Therefore,

\begin{equation*}
    v'(u)=0\hspace{.5cm}\stackrel{\det A(u)\neq 0}{\iff}\hspace{.5cm} \eta=\theta v(u)=-uA(u)^{-1}\theta\eta
\end{equation*}

and this last equation is equivalent to 

\begin{equation*}
    A\eta-u\theta\eta=A(u)\eta=-u\theta\eta,
\end{equation*}

from which we conclude that

\begin{equation*}
    v'(u)=0\hspace{.5cm}\iff\hspace{.5cm} A\eta=0.
\end{equation*}

Since we are assuming that $\det A\neq 0$ and $\eta\neq 0$, we conclude that $v'(u)\neq 0$, showing the regularity of the curve.

\subsubsection{Control sets with nonempty interior of $\Sigma_{\R^2}$}

In this section, we show the existence of control sets with nonempty interiors  for the control-affine system $\Sigma_{\R^2}$ introduced previously. We start with a result concerning the positive and negative orbits at equilibrium points of the system.

\begin{proposition}
\label{opensigma2}

If $\Sigma_{\R^2}$ satisfies the LARC then, with the exception of, at most, one $u\in\widehat{\Omega}$, the sets
$$\OC^+(v(u))\hspace{.5cm}\mbox{ and }\hspace{.5cm}\OC^-(v(u)),$$
are open sets.
 
\end{proposition}

\begin{proof}
Since the proofs for the positive and negative orbits are analogous, we show only the positive case. Moreover, since the positive orbit is positively invariant, $\OC^+(v(u))$ is open if and only if $v(u)\in \inner \OC^{+}(v(u))$. 

Let us now consider $s>0$ and $u\in\widehat{\Omega}$, and define the map

	$$f: \widehat{\Omega}^2\rightarrow\R^2, \;\;f(u_1, u_2):=\rme^{sA(u_2)}\left(\rme^{sA(u_1)}\left(v(u)-v(u_1)\right)+v(u_1)-v(u_2)\right)+v(u_2).$$
 
We observe that,

	$$f(u, u)=v(u)\;\;\mbox{ and  }\;\;f(u_1, u_2)=\phi(s, \phi(s, v(u), u_1), u_2)),\;\;\;\mbox{ implying that }\;\;\; f(\widehat{\Omega}^2)\subset\mathcal{O}^+(v(u)).$$

 As a consequence, $v(u)\in \inner\OC^+(v(u))$ if the differencial of $f$ is surjective on $(u, u)\in \widehat{\Omega}^2$. Calculating the partial derivatives of $f$, we obtain that

	$$\frac{\partial f}{\partial u_1}(u_1, u_2)=\rme^{sA(u_2)}\left(-s\theta\rme^{sA(u_1)}(v(u)-v(u_1))+(I_{\R^2}-\rme^{sA(u_1)})v'(u_1)\right)\hspace{.5cm}\mbox{ and},$$
	$$\frac{\partial f}{\partial u_2}(u_1, u_2)=-s\theta\rme^{sA(u_2)}\left(\rme^{sA(u_1)}(v(u)-v(u_1))+v(u_1)-v(u_2)\right)+(I_{\R^2}-\rme^{sA(u_2)})v'(u_2).$$
 
Evaluation at $(u, u)$ gives us,

	$$\frac{\partial f}{\partial u_1}(u, u)=\rme^{sA(u)}(I_{\R^2}-\rme^{sA(u)})v'(u)\hspace{1cm}\mbox{ and }\hspace{1cm}\frac{\partial f}{\partial u_2}(u, u)=(I_{\R^2}-\rme^{sA(u)})v'(u).$$

Let $R$ be a counterclockwise rotation of $\pi/2$. Then, 

$$\left\langle\frac{\partial f}{\partial u_1}(u, u), R\frac{\partial f}{\partial u_2}(u, u)\right\rangle=\left\langle \rme^{sA(u)}(I_{\R^2}-\rme^{sA(u)})v'(u), R(I_{\R^2}-\rme^{sA(u)})v'(u)\right\rangle$$
$$=\left\langle (I_{\R^2}-\rme^{sA(u)})\rme^{sA(u)}v'(u), R(I_{\R^2}-\rme^{sA(u)})v'(u)\right\rangle=\det(I_{\R^2}-\rme^{sA(u)})\left\langle \rme^{sA(u)}v'(u), R v'(u)\right\rangle,$$
where for the third equation we used that $\rme^{sA(u)}$ and $(I_{\R^2}-\rme^{sA(u)})$ commute and for the last one that $B(v, w)=\langle v,R w\rangle$ is an alternating bi-linear form. Consequently, 
$$
    \left\langle\frac{\partial f}{\partial u_1}(u, u), R\frac{\partial f}{\partial u_2}(u, u)\right\rangle\hspace{.5cm} \text{is}\hspace{.5cm}\text{L.I}\hspace{.5cm}\iff\hspace{.5cm} \det\left(1-\rme^{sA(u)}\right)\left\langle \rme^{sA(u)}v'(u), R v'(u)\right\rangle\neq 0.
$$
{\bf Claim 1:} For any $u\in\widehat{\Omega}$, there exists $\delta=\delta(u)>0$ such that $\det(I_{\R^2}-\rme^{sA(u)})\neq 0$ for all $s\in (0, \delta)$.

In fact, if for some $s_0>0$ and $u_0\in \widehat{\Omega}$ it holds that $\det(I_{\R^2}-\rme^{s_0A(u_0)})= 0$, there exists $v\neq 0$ satisfying $\rme^{s_0A(u_0)}v=v$. Now, if the eigenspace associated with the eigenvalue 1 has dimension 1, the fact that $A(u_0)$ and $\rme^{s_0A(u_0)}$ commute implies necessarily that $v$ is also an eigenvector of $A(u_0)$. However, $A(u_0)v=av$ implies that $$v=\rme^{s_0A(u)}v=\rme^{s_0a}v\hspace{.5cm}\implies\hspace{.5cm} a=0,$$
    contradicting the fact that  $u_0\in\widehat{\Omega}$. As a consequence,  $\rme^{s_0A(u_0)}=I_{\R^2}$ and from the relation $e^{s_0A(u_0)}=e^{s_0\operatorname{tr}A(u_0)}$, we obtain that $\operatorname{tr}A(u_0)=0$, allowing us to conclude that $A(u_0)$ has a pair of purely imaginary eigenvalues. Therefore, on some basis, 

\begin{equation*}
    A(u_0)=\left(\begin{array}{cc}
        0 & -a_0 \\
        a_0 & 0
    \end{array}\right), \hspace{.5cm} a_0\neq 0\hspace{.5cm}\implies\hspace{.5cm}\forall s\in\R, \hspace{.5cm}
    \rme^{sA(u_0)}=\left(\begin{array}{cc}
        \cos{sa_0} & -\sin{sa_0} \\
        \sin{sa_0} & \cos{sa_0}
    \end{array}\right),
\end{equation*}
and hence,
$$\det(I_{\R^2}-\rme^{s_0A(u_0)})=2(1-\cos{s_0a_0})=0\hspace{.5cm}\iff\hspace{.5cm} a_0s_0=\frac{\pi}{2}+k\pi, k\in\Z.$$
and so,
$$\det(I_{\R^2}-\rme^{sA(u_0)})\neq 0, \hspace{.5cm}\forall s\in \left(0, \delta\right), \hspace{.3cm}\mbox{ for any } 0<\delta<\frac{\pi}{2a_0},$$
showing the claim.

{\bf Claim 2:} With the exception of, at most, one $u\in\widehat{\Omega}$, there exists $\epsilon=\epsilon(u)$ such that
$$\left\langle \rme^{sA(u)}v'(u), R v'(u)\right\rangle\neq 0, \hspace{.5cm}\forall s\in (0, \epsilon).$$

Let us assume that $u_0\in\widehat{\Omega}$ does not satisfy the previous condition. In this case, 
\begin{equation}
\label{orthogonalvector}
\exists (s_n)_{n\in\mathbb{N}}\subset\R^+, \hspace{.4cm} \mbox{ with } \hspace{.4cm} s_n\rightarrow 0 \hspace{.4cm}\mbox{ and }\hspace{.4cm}
\left\langle \rme^{s_nA(u_0)}v'(u_0), R v'(u_0)\right\rangle=0.
\end{equation}

Hence,
$$
    0=\lim_{n\rightarrow\infty}\frac{1}{s_n}\left\langle \rme^{s_nA(u_0)}v'(u_0)-v'(u_0), R v'(u_0)\right\rangle=\left.\frac{d}{ds}\right|_{s=0}\left\langle \rme^{sA(u_0)}v'(u_0), R v'(u_0)\right\rangle=\left\langle A(u_0)v'(u_0), R v'(u_0)\right\rangle.
$$

In particular, 
$$A(u_0)v'(u_0)=\lambda v'(u_0),\hspace{.5cm}\mbox{ for some }\hspace{.5cm}\lambda\neq 0.$$

 Since $v'(u_0)=-A(u_0)^{-1}(\eta-\theta v(u_0))$, the previous gives us that

\begin{equation*}
    \eta-\theta v(u_0)=\lambda A(u_0)^{-1}(\eta-\theta v(u_0))\hspace{.5cm}\iff\hspace{.5cm}  A(u_0)(\eta-\theta v(u_0))=\lambda (\eta-\theta v(u_0)).
\end{equation*}

Remembering that $A(u_0)=A-u_0\theta$ and $A(u_0)v(u_0)=-u_0\eta$, allows us to obtain 
\begin{equation*}
    A\eta=A\eta-u_0\theta\eta+u_0\theta\eta=\lambda\eta-\lambda\theta v(u_0),
\end{equation*}

and, by applying $A(u_0)$ to the previous equality, we get that
\begin{equation*}
    A(u_0)A\eta=A(u_0)(\lambda\eta-\lambda\theta v(u_0))\hspace{.5cm}\stackrel{[A, A(u_0)]=0}{\iff}\hspace{.5cm} AA(u_0)\eta=\lambda A\eta - \lambda u_0\theta\eta + \lambda u_0\theta\eta=\lambda A\eta,
\end{equation*}
or, equivalently, 
\begin{equation*}
    A(A(u_0)\eta - \lambda \eta)=0.
\end{equation*}

Since we are assuming that $\det A\neq 0$, we get 

\begin{equation}\label{au}
    A(u_0)\eta=\lambda\eta \hspace{.5cm}\mbox{ or equivalently }\hspace{.5cm}\langle A(u_0)\eta,R\eta\rangle=0.
\end{equation}

Now, by the definition of $A(u)$, if $u_{1}\neq u_{2}$ satisfy (\ref{au}) then
$$\langle A\eta,R\eta\rangle=u_{1}\langle\theta\eta,R\eta\rangle\hspace{.5cm}\mbox{ and }\hspace{.5cm}\langle A\eta,R\eta\rangle=u_{2}\langle\theta\eta,R\eta\rangle,$$

and hence 
\begin{equation*}
    0=\langle A\eta,R\eta\rangle-\langle A\eta,R\eta\rangle=\underbrace{(u_1-u_2)}_{\neq 0}\langle\theta\eta,R\eta\rangle\hspace{.5cm},
\end{equation*}
implying that
\begin{equation*}
    \langle A\eta,R\eta\rangle=\langle\theta\eta,R\eta\rangle=0.
\end{equation*}

As a consequence, if (\ref{orthogonalvector}) holds for two different elements in $\Omega$, then necessarily $\eta$ is a common eigenvector of $A$ and $\theta$. From (\ref{LARC}) we conclude that $\Sigma_{\R^2}$ cannot satisfy the LARC which is a contradiction. Therefore, (\ref{orthogonalvector}) holds for, at most, one $u\in\widehat{\Omega}$, showing the claim.

According to claims 1. and 2. for all $u\in\widehat{\Omega}$, with the exception of at most one, there exists $\epsilon^* = \min(\delta, \epsilon)$, such that 
\begin{equation*}
    \det\left(I_{\R^2}-\rme^{sA(u)}\right)\left\langle \rme^{sA(u)}v'(u), R v'(u)\right\rangle\neq 0\hspace{.5cm}\forall s\in (0,\epsilon^*),
\end{equation*}
implying that the differential of $f$ at the point $(u, u)\in\widehat{\Omega}$ is surjective and concluding the proof.
	
\end{proof}

\begin{remark}
\label{scalar}
From the proof of the previous proposition, if the differential of the map $f$ at $(u_0, u_0)$ is not surjective, then $\langle A(u_0)\eta, R\eta\rangle=0$ or, equivalently, $\eta$ is an eigenvector of $A(u_0)$. 

Therefore, if $\Sigma_{S\setminus G(\theta)}$ satisfies the LARC, the previous imply that $A(u^*)$ is necessarily a scalar matrix, that is, 
    $$\langle A(u_0)\eta, R\eta\rangle=0\hspace{.5cm}\iff\hspace{.5cm} A(u_0)=\lambda(u_0)I_{\R^2}, \hspace{.5cm}\mbox{ with }\hspace{.5cm}\lambda(u_0)\neq 0.$$
\end{remark}

Next, we show the existence of control sets with a nonempty interior for the system $\Sigma_{\R^2}$ containing each connected component of $\widehat{\Omega}$ in its closure.

\begin{proposition}\label{controlaffine}

For any connected interval $I\subset\widehat{\Omega}$, there exists a control set $\CC_I$ of $\Sigma_{\R^2}$ such that $v(I)\subset \overline{\CC_I}$. Moreover, with the exception of, at most, one $u\in I$ it holds that $v(u)\in\inner \CC_I$.

\end{proposition}

\begin{proof} By Proposition \ref{opensigma2}, with the exception of at most one $u_0\in\widehat{\Omega}$, the sets $\OC^+(v(u))$ and $\OC^-(v(u))$ are open. As a consequence, 
$$\CC_u:=\overline{\OC^+(v(u))}\cap\OC^-(v(u))$$
is a control set of $\Sigma_{S\setminus G(\theta)}$ satisfying $v(u)\in\inner\CC_u.$ Let us denote by $u_0\in\widehat{\Omega}$ the point where Proposition \ref{opensigma2} fails and define the sets 
$$I^-:=\{u\in I; \;u<u_0\}\hspace{.5cm}\mbox{ and }\hspace{.5cm}I^+:=\{u\in I; \;u>u_0\}.$$

Since $I$ is a connect interval, the sets $I^{\pm}_0$ are also connect intervals, and hence their images $v(I^{\pm})$ are connected subsets of $\R^2$. Moreover, the fact that $v(u)\in\inner\CC_u$ for all $u\in I^{\pm}$ implies that 
\begin{equation}
    \label{equality}
    \CC_{u_1}=\CC_{u_2}\hspace{.5cm}\mbox{ for all }\hspace{.5cm}u_1, u_2\in I^{\pm}
\end{equation}

In fact, if $u_1, u_2\in I^+$ with $u_1<u_2$ satisfies 
$$v(u_2)\in\R^2\setminus\inner\CC_{u_1}\hspace{.5cm}\mbox{ then }\hspace{.5cm}v([u_1, u_2])\cap\partial \CC_{u_1}\neq\emptyset.$$
Therefore, there exists $\bar{u}\in I^+$ such that 
$v(\bar{u})\in\partial\CC_{u_1}$. 
On the other hand, 
$$\bar{u}\in I^+\hspace{.5cm}\implies\hspace{.5cm}v(\bar{u})\in\inner\CC_{\bar{u}}\hspace{.5cm}\implies\hspace{.5cm} \CC_{\bar{u}}\cap \CC_{u_1}\neq\emptyset\hspace{.5cm}\implies\hspace{.5cm} \CC_{\bar{u}}=\CC_{u_1}\hspace{.5cm}\implies\hspace{.5cm} v(\bar{u})\in\inner \CC_{u_1},$$
which is a contradiction. Therefore, relation (\ref{equality}) holds for $I^+$ and, analogously, for $I^-$. 

As a consequence, if $u_0\notin I$ we have nothing more to prove. If $u_0\in I$, to conclude the proof, we have to show that 
$$\exists u_1\in I^+, u_2\in I^-; \hspace{.5cm}\CC_{u_1}=\CC_{u_2}.$$

By Remark \ref{scalar}, it holds that $A(u_0)=\lambda(u_0)I_2$. In particular, the fact that $\lambda(u_0)^2=\det A(u_0)\neq 0$, implies the existence of $u_1<u_0<u_2$ such that $A(u_1)$ and $A(u_2)$ has a pair of eigenvalues with real parts of the same sign as $\lambda(u_0)$. Therefore, 
$$\phi(s, v(u_1), u_2)\rightarrow v(u_2)\hspace{.5cm}\mbox{ and }\hspace{.5cm} \phi(s, v(u_2), u_1)\rightarrow v(u_1),\hspace{.5cm} \lambda(u_0)s\rightarrow-\infty.$$
Since $v(u_i)\in\inner\CC_{u_i}$, $i=1, 2$, the previous allows us to construct a periodic chain passing through $\CC_{u_1}$ and $\CC_{u_2}$ which implies $\CC_{u_1}=\CC_{u_2}$ and concludes the proof.
\end{proof}

\begin{remark}
   As showed in \cite[Proposition 3.8]{DSAyAOR}, the equilibrium $v(u_0)$ associated with the control $u_0\in\widehat{\Omega}$, where Proposition \ref{opensigma2} fails, can be in the boundary of the control set $\CC_{I}$.
\end{remark}

The next result states some topological properties of the control sets $\CC_I$.

\begin{proposition}
\label{propertiesR2}
Let $I\subset\widehat{\Omega}$ be a connected interval, and assume that $\det A(u)>0$ for all $u\in I$. Then, for the control set $\CC_I$, it holds that:
\begin{itemize}
    \item[1.] $\tr A(u)>0$ for all $u\in I$ and $\CC_I$ is open;
    \item[2.] $\tr A(u)<0$ for all $u\in I$ and $\CC_I$ is closed;
     \item[3.] $\tr A(u)=0$ for some $u\in I$ and $\CC_I=\R^2$;
   
\end{itemize}

\end{proposition}

\begin{proof} Assume that $A(u)$ has a pair of eigenvalues with negative real parts. Then, for any $v\in\R^2$ we have that 
$$\phi(s, v, u)=\rme^{sA(u)}(v-v(u))+v(u)\rightarrow v(u), \hspace{.5cm}s\rightarrow+\infty.$$
If $\OC^-(v(u))$ is open, there exists $s_0>0$ such that 
$$\phi(s_0, v, u)\in\OC^-(v(u))\hspace{.5cm}\implies\hspace{.5cm}v\in\phi\left(-s_0, \OC^-(v(u)), u\right)\subset \OC^-(v(u)),$$
implying that $\OC^-(v(u))=\R^2$. In same way, if $A(u)$ has a pair of eigenvalues with positive real parts and $\OC^+(v(u))$ is open, then $\OC^+(v(u))=\R^2$. 

Since we are assuming $\det A(u)>0$, if $\tr A(u)\neq 0$ for all $u\in I$, then necessarily $A(u)$ has a pair of eigenvalues with positive real parts if $\tr A(u)>0$ and negative real parts if $\tr A(u)<0$, which by the previous arguments implies items 1. and 2.

On the other hand, since $\tr A(u)=\tr A+u\tr\theta$, if $\tr A(u)=0$ for some $u\in I$, we have the following possibilities:

$\bullet$\;  If $\tr A(u_0)=0$ and $\tr A(u_1)\tr A(u_2)<0$ for any  $u_1, u_2\in\Omega_0$ satisfying $u_1<u_0<u_2$. 

In this case, there exist $u_1, u_2\in I$ where Proposition \ref{opensigma2} holds, such that $A(u_1)$ has a pair of eigenvalues with positive real parts and $A(u_2)$ has a pair of eigenvalues with negative real parts. As a consequence,
$$\R^2=\OC^+(v(u_1))\hspace{.5cm}\mbox{ and }\hspace{.5cm}\OC^-(v(u_2))=\R^2,$$
implying that 
$$\OC^-(v(u_1))=\overline{\OC^+(v(u_1))}\cap\OC^-(v(u_1))=\CC_I=\overline{\OC^+(v(u_2))}\cap\OC^-(v(u_2))=\overline{\OC^+(v(u_2))}.$$
Therefore, $\CC_I$ is open and closed in $\R^2$ which implies $\CC_I=\R^2$.

$\bullet$ \; $\tr A(u)=0$ for all $u\in I$\footnote{The proof for this case is analogous to the one in \cite[Theorem 3.6]{DSAyAOR}. However, for completeness sake, we reproduce it here.};
    
In this case, we have that $A(u)$ has a pair of pure imaginary eigenvalues for all $u\in I$. Moreover, the fact that $A$ and $\theta$ commute with $A(u)$, implies necessarily that 
$$A=\left(\begin{array}{cc}
    0 & -\mu \\
   \mu  & 0
\end{array}\right)\hspace{.5cm}\mbox{ and }\hspace{.5cm}\theta=\left(\begin{array}{cc}
    0 & -1 \\
   1  & 0
\end{array}\right)\hspace{.5cm}\implies\hspace{.5cm} A(u)=\left(\begin{array}{cc}
    0 & -(\mu-u) \\
   \mu-u  & 0
\end{array}\right),$$
with $\mu\neq 0$ and $\mu\notin I$. Therefore, for any $u\in I$, the solutions of $\Sigma_{S\setminus G(\theta)}$ are given by concatenations of the curves
	$$\phi(s, v, u)=R_{s(\mu-u)}(v-v(u))+v(u), \;\;\;\mbox{ where }\;\;\; v(u)=-uA(u)^{-1}\eta=\frac{u}{\mu-u} R\eta,$$
 and $R_{s(\mu-u)}$ stands for the rotation of $s(\mu-u)$-degrees. In particular, 
	$$|\phi(s, v, u)-v(u)|=|R_{s(\mu-u)}(v-v(u))|=|v-v(u)|,$$
  shows that the solution curve $s\mapsto\varphi(s, v, u)$ lies onto the circumference $C_{u, v}$ with radius $|v-v(u)|$ and center $v(u)$.
  
	In order to show the controllability of $\Sigma_{\R^2}$, we will construct a periodic trajectory from an arbitrary point $v_0\in\R^2$ to a point $v(u_0)\in v(I)$ with $u_0\in\inner I$ as follows:
	
	\begin{itemize}
		\item[(a)] Let $u_1, u_2\in I$ with $u_1<u_0<u_2$, which is possible since $u_0\in\inner I$.
		
		\item[(b)] The circumference $C_{u_2, v_0}$ intersects the line $\R\cdot R\eta$ at two points. Let us denote by $v_1$ the point of this intersection that is closer to $v(u_1)$ and consider $s_1>0$ such that $v_1=\phi(s_1, v, u_2)$;
		
		\item[(c)] If $v_1\notin v([u_1, u_2])$, we repeat item (a) with the circumference $C_{u_1, v_1}$, in order to obtain a point 
  $$v_2=\phi(s_2, v_1, u_1)\in C_{u_1, v_1}\cap \R\cdot R\eta.$$ 
		
		\item[(d)] Inductively, if a point $v_n\notin v([u_1, u_2])$ was obtained by the previous process, we define the point $v_{n+1}$ in the intersection of the circumference $C_{u_i, v_n}$ with the line $\R\cdot R\eta$, where $i=1$ if $n$ is even and $i=2$ if $n$ is odd. Moreover, the radius $r_n$ of the circumference $C_{u_i, v_n}$ satisfies 
		$$r_n=|v_n-v(\nu_i)|=|v_0-v(\nu_2)|-n\cdot |v(\nu_2)-v(\nu_1)|.$$
        
		Therefore, for some $N\in\N$ large enough, we obtain that $v_N\in v(I)\subset \R\cdot R\eta$.
		
		\item[(e)] Since $v_N\in v(I)$, the continuity of the curve $u\mapsto v(u)$ assures the existence of $u_N\in I$ satisfying $|v(u_N)-v_N|=|v(u_N)-v(u_0)|.$ As a consequence, the circumference $C_{u_N ,v_N}$ contains $v_N$ and $v(u_0)$, and hence, there exists $s_{N+1}>0$ such that $\varphi(s_{N+1}, v_N, u_N)=v(u_0)$. By concatenation, we obtain a trajectory from $v_0$ to $v(u_0)$.
		
		\item[(f)] By choosing the ``complementary half" of the circumferences $C_{u_i, v_n}$, $n=0, \ldots N-1$ and $C_{u_N, v_N}$, obtained in the previous items, we obtain a trajectory from $v(u_0)$ to $v_0$.
		
	\end{itemize}
	
	The previous steps show how to construct a periodic trajectory passing through $v_0$ and $v(u_0)$. By the arbitrariness of $v_0\in\R^2$ we conclude that $C_I=\R^2$, showing that $\Sigma_{\R^2}$ is controllable and concluding the proof.
\end{proof}	

\subsection{The control sets of LCSs with nilrank two on $G(\theta)$}

In this section, we prove our main results by characterizing the control sets of LCSs on $G(\theta)$ with nilrank two. The idea is basically to use Proposition \ref{conjugation3} and the results in the previous sections. We start with the following result characterizing the control sets of the system $\Sigma_{\R\times\R^2}$.

\begin{proposition}
\label{controlsetsigma1}
Let us consider the control system
 \begin{flalign*}
	  &&\left\{\begin{array}{l}
     \dot{t}=u\alpha \\
     \dot{v}=(A-u\alpha\theta)v+u\eta
\end{array}\right., \; u\in\Omega. &&\hspace{-1cm}\left(\Sigma_{\R\times\R^2}\right)
	  \end{flalign*}
If $\Sigma_{\R\times\R^2}$ satisfies the LARC, then it admits a control set with a nonempty interior $\CC_{\R\times\R^2}$  satisfying
$$\CC_{\R\times\R^2}=\R\times\CC_{\R^2},$$
where $\CC_{\R^2}$ is the control set with nonempty interior of the associated control-affine system 
 \begin{flalign*}
	  &&\dot{v}=A(u)v+\alpha u\eta, \; u\in\Omega.&&\hspace{-1cm}\left(\Sigma_{\R^2}\right)
	  \end{flalign*}
satisfying $0\in\overline{\CC_{\R^2}}$.
    
\end{proposition}

\begin{proof}
Let us first notice that the solutions of system $\Sigma_{\R\times\R^2}$ satisfy 
$$\phi(s, (t_0,v_0), {\bf u})=(\phi_1(s, t_0, {\bf u}), \phi_2(s, v_0, {\bf u})),\hspace{.5cm}\forall s\in\R, (t_0, v_0)\in\R\times\R^2, {\bf u}\in\UC,$$
and hence, they are given by  concatenations of the solutions
$$\phi(s, (t_0,v_0) ,u)=(t_0+ u\alpha s, \rme^{sA(u)}(v_0-v(u))+v(u)),  .$$

Let us assume, w.l.o.g., that $\alpha=1$, otherwise, we change the control range $\Omega$ by $\alpha\Omega$. Let us show that $\CC_{\R\times\R^2}$ satisfies the three conditions in Definition \ref{controlset}.

\begin{itemize}
    \item[(i)] {\it Weak invariance:} Let $(t_0,v_0)$ be a point in $\CC_{\R\times\R^2}$. Then, $v_0\in\CC_{\R^2}$ and there exist ${\bf u}\in\UC$ such that $$\phi_2(\R^+, v_0, {\bf u})\hspace{.5cm}\implies\hspace{.5cm}\phi(\R^+, v_0, {\bf u})\subset\R\times\CC_{\R^2}=\CC_{\R\times\R^2}.$$
    \item[(ii)] {\it Approximate controllability:} Let us start by showing that exact controllability holds in $\R\times\inner\CC_{\R^2}$. Let then $(t_1,v_1), (t_2,v_2)\in\R\times\inner\CC_{\R^2}$ and consider $v(u_1), v(u_2)\in\inner\CC_{\R^2}$ with $u_1<0<u_2$.

    Since controllability holds in $\inner\CC_{\R^2}$, there exists $t'_1, t'_2\in\R$, ${\bf u}_1, {\bf u}_2\in\UC$ and $s_1, s_2>0$ such that 
$$\phi(s_1, (t_1, v_1), {\bf u}_1)=(t'_1, v(u_1))\hspace{.5cm}\mbox{ and }\hspace{.5cm} \phi(s_2, (t'_2, v(u_2)), {\bf u}_2)=(t_2, v_2).$$
    Analogously, there exists $s_3>0$ and ${\bf u}_3\in\UC$ such that 
    $$\phi(s_3, (t_1', v(u_1)), {\bf u}_3)=(t''_2, v(u_2)).$$

    $\bullet$ If $t''_2\leq t_2'$ we have that 
    $$s_4=\frac{t_2'-t_2''}{u_2\alpha}\geq 0\hspace{.5cm}\mbox{ and }\hspace{.5cm}\phi(s_4, (t_2'', v(u_2), u_2)=(t_2''+u_2\alpha s_4, v(u_2))=(t_2', v(u_2)).$$
    Hence, by concatenation
    $$\phi(s_2, \phi(s_4, \phi(s_3, \phi(s_1, (t_1, v_1), {\bf u}_1), {\bf u}_3), u_2), {\bf u}_2)=(t_2, v_2).$$
    
    $\bullet$ If $t''_2> t_2'$ we have that 
    $$s'_4=\frac{t_2'-t_2''}{\alpha u_1}>0\hspace{.5cm}\mbox{ and }\hspace{.5cm}\phi(s'_4, (t_1', v(u_1), u_1)=(t_1'+u_1\alpha s'_4, v(u_1))=(t_1'+t_2'-t_2'', v(u_1)),$$
    implying that 
    $$\phi(s_3, \phi(s'_4, (t_1', v(u_1), u_1), {\bf u}_3)=\phi(s_3, (t_1'+t_2'-t_2'', v(u_1)), {\bf u}_3)$$
    $$=(t_2'-t_2'', 0)+\phi(s_3, (t_1', v(u_1)), {\bf u}_3)=(t_2'-t_2'', 0)+(t_2'', v(u_2))=(t_2', v(u_2)).$$
    Therefore, by concatenation, 
    $$\phi(s_2, \phi(s_3, \phi(s'_4, \phi(s_1, (t_1, v_1), {\bf u}_1), u_1), {\bf u}_3), {\bf u}_2)=(t_2, v_2),$$
    showing that controllability holds inside $\R\times\inner\CC_{\R^2}$ (see Figure \ref{figura4}).

\begin{figure}[!h]
	\centering
	\includegraphics[scale=.8]{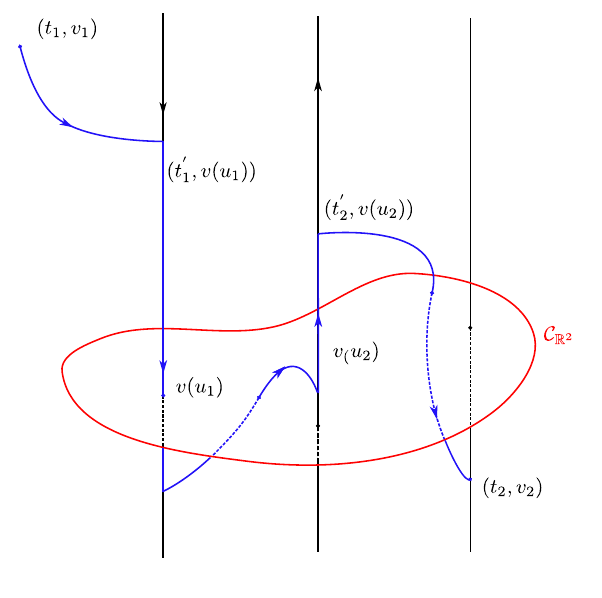}
	\caption{Trajectory connecting points in the interior of $\R\times\CC_{\R^2}$.}
	\label{figura4}
\end{figure}

    As a consequence, we get that 
    $$\forall (t, v)\in\R\times\inner\CC_{\R^2}, \hspace{.5cm}\R\times\inner\CC_{\R^2}\subset\OC^+(t, v)\hspace{.5cm}\implies\hspace{.5cm}\R\times\CC_{\R^2}\subset\overline{\OC^+(t, v)}.$$
    On the other hand, if $v\in\CC_{\R^2}$ and $v_0\in\inner\CC_{\R^2}$ we have that 
    $$v\in\OC_2^{-}(v_0)\hspace{.5cm}\iff\hspace{.5cm}v_0\in\OC_2^+(v).$$
    Therefore, for any $(t, v)\in\R\times\CC_{\R^2}$ there exists $(t_0, v_0)\in\R\times\CC_{\R^2}$ such that $(t_0, v_0)\in\OC^+(t, v)$ implying by the previous that 
    $$\R\times\CC_{\R^2}\subset \overline{\OC^+(t_0, v_0)}\subset\overline{\OC^+(t, v)},$$
    showing that $\R\times\CC_{\R^2}$ satisfies condition 2. in the Definition \ref{controlset}.

     \item {\it Maximality:} If $D\subset\R\times\R^2$ is a control set such that $\R\times\CC_{\R^2}\subset D$, then $\pi_2(D)$ is contained in a control set of $\Sigma_{\R^2}$. Since $\CC_{\R^2}\subset\pi_2(D)$ we have by maximality that $\CC_{\R^2}=\pi_2(D)$ and hence $$D\subset\pi_2^{-1}(\pi_2(D))=\pi_2^{-1}(\CC_{\R^2})=\R\times\CC_{\R^2},$$
     which shows that $\CC_{\R\times\R^2}=\R\times\R^2$ is a control set of $\Sigma_{\R\times\R^2}$.
\end{itemize}  

\end{proof}

\begin{remark}
    Let us note that Proposition \ref{conjugation} cannot be used here since we do not know, a priori, that a control set for the system $\Sigma_{\R\times\R^2}$ exists. 
\end{remark}

Next, we state and prove our main result for linear control systems with nilrank two.

\begin{theorem}
\label{mains}
    Any linear control system  $\Sigma_{G(\theta)}$ on $G(\theta)$ with nilrank two satisfying the LARC admits a unique control set given by 
    $$\CC_{G(\theta)}=\pi^{-1}(\CC_{G^0\setminus G(\theta)}), \hspace{.5cm}\mbox{ where }\hspace{.5cm} \pi: G(\theta)\rightarrow G^0\setminus G(\theta),$$
    is the canonical projection, $G^0$ is the set of singularities of $\XC$ and $\CC_{G^0\setminus G(\theta)}$ the control set of the induced system $\Sigma_{G^0\setminus G(\theta)}$ on $G^0\setminus G(\theta)$ containing the origin in its closure. Moreover, if $\XC=(A, \xi)$ is the associate linear vector field,
    $$\widehat{\Omega}=\{u\in\Omega; \;\;\det (A+\alpha u\theta)\neq 0\},$$
    and $I_0\subset\widehat{\Omega}$ the connected component of $0$
    we get that:
\begin{itemize}
    \item[1.] If $\det A>0$ and $\tr(A-\alpha 
 u\theta)>0$ for all $u\in I_0$, then $\CC_{G(\theta)}$ is open;  
    \item[2.] If $\det A>0$ and $\tr (A-\alpha 
 u\theta)<0$ for all $u\in I_0$, then $\CC_{G(\theta)}$ is closed;
    \item[3.] If $\det A>0$ and $\tr (A-\alpha 
 u\theta)=0$ for some $u\in I_0$, then $\CC_{G(\theta)}=G(\theta)$.
\end{itemize}
    
\end{theorem}

\begin{proof} Since $\Sigma_{G(\theta)}$ has nilrank two, we get that $\det A\neq0$, and hence, Proposition \ref{conjugation3} implies the existence of a diffeomorphism $\psi:G(\theta)\rightarrow \R\times\R^2$, fixing the identity element $e=(0,0)$, and conjugating  $\Sigma_{G(\theta)}$ to the control-affine system 
\begin{flalign*}
	  &&\left\{\begin{array}{l}
     \dot{t}=u\alpha \\
     \dot{v}=(A-u\alpha\theta)v+u\eta
\end{array}\right., \; u\in\Omega. &&\hspace{-1cm}\left(\Sigma_{\R\times\R^2}\right)
\end{flalign*}
By the previous proposition $\Sigma_{\R\times\R^2}$ admits a control set $\CC_{\R\times\R^2}$ with nonempty interior and containing $v(\Omega_0)$, implying that  $$\CC_{G(\theta)}:=\psi^{-1}(\CC_{\R\times\R^2}),$$
is a control set of $\Sigma_{G(\theta)}$ containing the identity element in its closure. Moreover, by Proposition \ref{controlsetsigma1}, it holds that $\CC_{\R\times\R^2}=\R\times\CC_{\R^2}=\pi_2^{-1}(\CC_{\R^2}),$
implying that
$$\CC_{G(\theta)}=\psi^{-1}(\CC_{\R\times\R^2})=\psi^{-1}(\pi_2^{-1}(\CC_{\R^2}))=(\pi_2\circ\psi)^{-1}(\CC_{\R^2})=(\widetilde{\psi}\circ\pi)^{-1}(\CC_{\R^2})=\pi^{-1}(\widetilde{\psi}^{-1}(\CC_{\R^2}))=\pi^{-1}(\CC_{G^0\setminus G(\theta)}),$$
where $\widetilde{\psi}:G^0\setminus G(\theta)\rightarrow \R^2$ is the diffeomorphism that conjugates $\Sigma_{G^0\setminus G(\theta)}$ and the control-affine system $\Sigma_{\R^2}$ given by Proposition \ref{conjugation3}. 

The properties 1., 2. and 3. follows from Proposition \ref{propertiesR2}, and the previous relations. Therefore, it only remains to show the uniqueness of $\CC_{G(\theta)}$. 

If $e\in\inner\CC_{G(\theta)}$, Theorem \ref{subgrupos} together with the solubility of $G(\theta)$ imply the uniqueness of $\CC_{G(\theta)}$. 
On the other hand, if $e\notin\inner C_{G(\theta)}$ we have, by conjugation, that $0\in \widehat{\Omega}$ is the only point where Proposition \ref{opensigma2} fails. In particular, we must have that 
$\langle A\eta, R\eta\rangle=0$ and hence $A=\lambda I_{\R^2}$, for some $\lambda\neq 0$ (see Remark \ref{scalar}). If $D\subset\R\times\R^2$ is a control set with a nonempty interior of $\Sigma_{\R\times\R^2}$ for any $(t, v)\in\inner D$ there exists $S>0$ and ${\bf u}\in\UC$ satisfying 
$$\phi(S, (t, v), {\bf u})=(t, v)\hspace{.5cm}\mbox{ where }\hspace{.5cm} S=\sum_{j=1}^n s_j \hspace{.5cm}\mbox{ with }\hspace{.5cm}0=s_0<s_1<\ldots<s_n,$$
and ${\bf u}_{|[s_j, s_{j+1})}=u_j\in\Omega$ for $j=0, \ldots, n-1.$
Therefore, 
$$\sum_{i=1}^n s_i u_i=0\hspace{.5cm}\mbox{ and, by equation (\ref{formula}) }\hspace{.5cm} \phi_2(S,v,\textbf{u})=\rme^{\sum_{i=1}^{n}s_iA(u_i)}v+\phi_2(s, 0, {\bf u})=v.$$
Moreover, 
$$\sum_{i=1}^{n}s_iA(u_i)=\sum_{i=1}^{n}s_i(A-u_i\theta)=\tau A-\underbrace{\left(\sum_{i=1}^n s_i u_i\right)}_{=0}\theta=\tau A,$$
implies that 
$$\phi_2(\tau,v,\textbf{u})=\rme^{\tau A}v+\phi_2(\tau, 0, {\bf u})=v.$$

Let us assume, w.l.o.g., that $\lambda<0$, since the other possibility is analogous. In this case, the periodicity of $(t, v)$ implies that
$$v=\phi_2(n\tau, v, {\bf u})=\rme^{n\tau\lambda}v+\phi_2(n\tau, 0, {\bf u}),$$
and then
$$\rme^{n\tau \lambda}\rightarrow 0\hspace{.5cm}\implies\hspace{.5cm}\phi_2(n\tau, 0, {\bf u})\rightarrow v\in\pi_2(\inner D).$$
Since $0\in\overline{\CC_{\R^2}}$, there exists by continuity $v_1\in\inner\CC_{\R^2}$ such that $\phi_2(n\tau, v_1, {\bf u})\in \pi_2(\inner D)$. On the other hand, for any $u\in\widehat{\Omega}$ sufficiently close to $0$, we have that $A(u)$ still has a pair of eigenvalues with negative real parts and $v(u)\in\inner \CC_{\R^2}$. Hence, 
$$\phi_2(s, v, u)\rightarrow v(u), \hspace{.5cm}s\rightarrow+\infty\hspace{.5cm}\implies \hspace{.5cm}\exists s_0>0, \hspace{.5cm}\phi_2(s_0, v, u)\in \inner\CC_{\R^2}.$$
Since, by Proposition \ref{conjugation}), $\pi_2(\inner D)$ has to be contained in the interior of a control set $D'\subset\R^2$ of $\Sigma_{\R^2}$, the previous imply (by exact controllability in $\inner D'$) the existence of an orbit starting and finishing in $\inner\CC_{\R^2}$ and passing by $\inner D'$, forcing the equality  $D'=\CC_{\R^2}$. Therefore, 
$$D\subset\pi_2^{-1}\pi_2(D)\subset \pi_2^{-1}(\CC_{\R^2})=\R\times \CC_{\R^2}=\CC_{\R\times\R^2},$$
showing the uniqueness of $\CC_{\R\times\R^2}$ and hence, $\CC_{G(\theta)}$ is the unique control set with a nonempty interior of $\Sigma_{G(\theta)}$, concluding the proof.
\end{proof}

\section{LCSs with nilrank smaller than two on $G(\theta)$}

In this section, the cases where the nilrank of a LCS is equal to zero or one are analyzed. As we will see, for nilrank one control sets, the LARC is enough to assure the existence of control sets with a nonempty interior. On the other hand, the dynamics of LCSs with nilrank zero strongly depend on the group structure.

\subsection{Nilrank one LCSs}

We start by proving an equivalence between the LARC and the ad-rank condition for a LCS with nilrank one.

\begin{lemma}
For a LCS with nilrank one, the LARC is equivalent to the ad-rank condition.
\end{lemma}

    \begin{proof}
    Up to isomorphism, we can consider $\Sigma_{G(\theta)}$ to be a LCS on $G(\theta)$ satisfying the LARC and such that $\XC=(A, \xi)$ and $Y=(\alpha, 0)$, with $\alpha\neq 0$. Hence, by equations (\ref{LARC}) and (\ref{adrank}),  
    $$\mbox{LARC is equivalent to } \langle A\xi, R\xi\rangle^2+\langle \theta\xi, R\xi\rangle^2\neq 0,$$
    and 
    $$\mbox{ad-rank is equivalent to } \langle A\xi, R\xi\rangle\neq 0.$$
    However, the assumption that $\Sigma_{G(\theta)}$ has nilrank one implies necessarily that $\dim\ker A=1$. As a consequence, we have the following possibilities:
    \begin{enumerate}
        \item $\xi\in\ker A$ which by the commutativity of $A$ and $\theta$ imply that $\langle \theta\xi, R\xi\rangle=0$ and hence LARC holds only if the ad-rank condition holds;

        \item $\xi\notin\ker A$ and hence $\R\xi$ is a one-dimensional eigenspace of $A$ and $\R^2=\R\xi\oplus\ker A$. Consequently, 
        $$\langle A\xi, R\xi\rangle=0\hspace{.5cm}\implies\hspace{.5cm}\langle \theta\xi, R\xi\rangle=0,$$ 
        showing again that the LARC implies the ad-rank condition.
    \end{enumerate}

    In both cases, the LARC implies the ad-rank condition, proving the result.
    \end{proof}

\bigskip

As a consequence of the previous result and Theorem \ref{subgrupos}, any LCS with nilrank one admits a unique control set $\CC_{G(\theta)}$ with nonempty interior, satisfying $G^0\subset \inner\CC_{G(\theta)}$. 

On the other hand, the assumption that the nilrank of $\Sigma_{G(\theta)}$ is equal to one implies that $H:=\{0\}\times\ker A\subset G(\theta)$ is a closed, normal, and one-dimensional subgroup. Consequently, we have a well-induced linear control system on the 2D Lie group $H \setminus G(\theta)$ and the following holds:

\begin{theorem}
    Any LCS with nilrank one satisfying the LARC admits a unique control set with a nonempty interior $\CC_{G(\theta)}$ satisfying 
    $$\CC_{G(\theta)}=\pi^{-1}(\CC_{H\setminus G(\theta)})= \CC_{H\setminus G(\theta)}\times \ker A,$$ 
        where $\pi: G(\theta)\rightarrow H\setminus G(\theta)$ is the canonical projection and $\XC=(A, \xi)$. 
    Moreover, $\CC_{G(\theta)}$ is open if $\tr A>0$ and closed if $\tr A<0$, or $\tr A=0$ and $\CC_{G(\theta)}=G(\theta)$.
\end{theorem}

\begin{proof} Since the induced LCS $\Sigma_{H\setminus G(\theta)}$ on $H\setminus G(\theta)$ satisfies the LARC, there exists a unique control set $\CC_{H\setminus G(\theta)}$ that contains the identity element of $H\setminus G(\theta)$ in its interior (Theorem \ref{subgrupos}). Moreover, 
$$\pi^{-1}(eH)=H=\{0\}\times\ker A\subset\inner\CC_{G(\theta)},$$
since $H$ is contained in the set of fixed points of the flow of $\XC=(A, \xi)$. By Proposition \ref{conjugation} we obtain that 
$\CC_{G(\theta)}=\pi^{-1}(\CC_{H\setminus G(\theta)})$ proving the first equality.

On the other hand, if $w\in\ker A$ is a unitary vector and we consider the basis  $\{(1, 0), (0, Rw), (0, w)\}$, the canonical projection is given by
$$\pi:G(\theta)\rightarrow (t, \langle v, Rw\rangle),$$
which coincides with projection onto the first two components of the chosen basis.  As a consequence, the control set $\CC_{H\setminus G(\theta)}$ is naturally identified inside $G(\theta)$ as the subset $\CC_{H\setminus G(\theta)}\times\{0\}$ of the plane generated by $\{(1, 0), (0, Rw)\}$. Since 
$$(0, s_2 w)(t, s_1Rw)=(t, s_2 w+ s_1 Rw)=t(1, 0)+s_1(0, Rw)+s_2(0, w), \hspace{.5cm}\forall t, s_1, s_2\in\R,$$
we get that 
$$\pi^{-1}(\CC_{H\setminus G(\theta)})=H\left(\CC_{H\setminus G(\theta)}\times\{0\}\right)= \CC_{H\setminus G(\theta)}\times \ker A,$$
proving the second equality. Now, since $\pi$ conjugates $\Sigma_{G(\theta)}$ and $\Sigma_{H\setminus G(\theta)}$, it commutes with the associated derivations, that is, 
$$\hat{\DC}\circ\pi=\pi\circ\DC,$$
where $\hat{\DC}$ is the derivation associated with the system $\Sigma_{H\setminus G(\theta)}$. Since 
$$\{0\}\times\ker A\subset\ker\DC\hspace{.5cm}\mbox{ and }\hspace{.5cm}\dim\ker\DC=2,$$
we obtain that $\tr\hat{\DC}=\tr \DC=\tr A$ is the only possible eigenvalue of $\hat{\DC}$, which by Theorem \ref{subgrupos} implies the result.
\end{proof}

\subsection{Nilrank zero LCSs}

In this section, we analyze the LCSs with nilrank zero. As we will see, the LARC is not enough to assure the existence of control sets with nonempty interiors for LCSs in some of the classes of the groups $G(\theta)$. Instead, they admit an infinite number of control sets with empty interior. 

We start with a lemma that will be used in the proof of the results of this section.

\begin{lemma}
\label{positive}
Let us assume that $\theta\neq\left(\begin{array}{cc}
    \gamma & -1 \\
     1 & \gamma
\end{array}\right)$ for all $\gamma\neq 0$. 
If $\langle \theta\xi, R\xi\rangle\neq 0$, there exists $\hat{\xi}\in\R^2$ such that 
    $$g(t):=\langle\Lambda_t^{\theta}\xi, \hat{\xi}\rangle\geq 0, \hspace{1cm}\forall t\in\R.$$
\end{lemma}

\begin{proof}
    The lemma is proved case by case.

    1. $\theta=\left(\begin{array}{cc}
    1 & 0 \\
    0 & \gamma
\end{array}\right)$, $|\gamma|\in (0, 1]$: In this case, $\langle \theta\xi, R\xi\rangle\neq 0$ implies that $\xi=\xi_1 \mathbf{e}_1+\xi_2 \mathbf{e}_2$ with $\xi_1\xi_2\neq 0$.

Let us assume that $\gamma\in (0, 1)$, since the case $\gamma\in (-1, 0)$ is analogous. Consider $\hat{\xi}=\gamma\xi_1^{-1} \mathbf{e}_1-\gamma\xi_2^{-1}\mathbf{e}_2$, we obtain 
$$g(t)=\langle\Lambda_t^{\theta}\xi, \hat{\xi}\rangle=\left\langle (\rme^t-1)\xi_1 \mathbf{e}_1+\frac{1}{\gamma}(\rme^{\gamma t}-1)\mathbf{e}_2, \gamma\xi_1^{-1} \mathbf{e}_1-\gamma\xi_2^{-1}\mathbf{e}_2\right\rangle=\gamma(\rme^t-1)-(\rme^{\gamma t}-1).$$
Derivation implies that $g'(t)=\gamma(\rme^{t}-\rme^{\gamma t}),$ and hence
$$g'(t)>0, \hspace{.5cm}t\in (0, +\infty) \hspace{.5cm}\mbox{ and }\hspace{.5cm}g'(t)<0, \hspace{.5cm}t\in (-\infty, 0),$$
showing that $g$ is strictly decreasing in $(-\infty, 0)$ and strictly increasing in $(0, +\infty)$. Since $g(0)=0$ we obtain that $g(t)\geq 0$ for all $t\in\R$.

\bigskip

2. $\theta=\left(\begin{array}{cc}
    1 & 0 \\
    0 & 0
\end{array}\right)$: In this case, $\langle \theta\xi, R\xi\rangle\neq 0$ implies also that $\xi=\xi_1 \mathbf{e}_1+\xi_2 \mathbf{e}_2$ with $\xi_1\xi_2\neq 0$. Defining $\hat{\xi}= \xi_1^{-1} \mathbf{e}_1-\xi_2^{-1}\mathbf{e}_2$ gives us
$$g(t)=\langle\Lambda_t^{\theta}\xi, \hat{\xi}\rangle=\left\langle (\rme^t-1)\xi_1 \mathbf{e}_1+t \xi_2\mathbf{e}_2, \xi_1^{-1} \mathbf{e}_1-\xi_2^{-1}\mathbf{e}_2\right\rangle=(\rme^t-t-1),$$

As in the case before, deriving the function, we obtain $g'(t)=\rme^t-1$, which satisfies

$$g'(t)>0, \hspace{.5cm} t\in (0, +\infty) \hspace{.5cm}\mbox{ and }\hspace{.5cm}g'(t)<0, \hspace{.5cm}t\in (-\infty, 0),$$

and as in the previous case, $g(t)\geq g(0)=0$ for all $t\in\R$.

\bigskip

3. $\theta=\left(\begin{array}{cc}
    1 & 1 \\
    0 & 1
\end{array}\right)$: In this case, $\langle \theta\xi, R\xi\rangle\neq 0$ implies that $\xi=\xi_1 \mathbf{e}_1+\xi_2 \mathbf{e}_2$ with $\xi_2\neq 0$. By considering

and we should define $\hat{\xi}=\xi_2^{-1}\mathbf{e}_1-\xi_1\xi_2^{-2}\mathbf{e}_2.$ A simple calculation shows us that 
$$g(t)=\langle\Lambda_t^{\theta}\xi, \hat{\xi}\rangle=-\rme^{t}+1+t\rme^t.$$
As in the first case, $g'(t)=te^t$ and hence
$$g'(t)>0, \hspace{.5cm}t\in (0, +\infty) \hspace{.5cm}\mbox{ and }\hspace{.5cm}g'(t)<0, \hspace{.5cm} t\in (-\infty, 0),$$
implying again that  $g(t)\geq g(0)=0$, forall $t\in\R$.

\bigskip

4. $\theta=\left(\begin{array}{cc}
    0 & -1 \\
    1 & 0
\end{array}\right)$: In this case, $\langle \theta\xi, R\xi\rangle\neq 0$ implies $\xi\neq 0$. Defining $\hat{\xi}= -\xi_2\mathbf{e}_1+\xi_1\mathbf{e}_2$ gives us
$$g(t)=\langle\Lambda_t^{\theta}\xi, \hat{\xi}\rangle=\Bigl\langle \left[\xi_2(\cos t-1)+\xi_1\sin t \right]\mathbf{e}_1+\left[\xi_2\sin t-\xi_1(\cos t-1) \right]\mathbf{e}_2,-\xi_2\mathbf{e}_1+\xi_1\mathbf{e}_2\Bigr\rangle$$
$$=-\left(\xi_1^2+\xi_2^2\right)(\cos t-1)=|\xi|^2(1-\cos t)\geq 0,$$
as required.

\end{proof}

Using the previous lemma, we are able to show the following:

\begin{theorem}\label{kerA=2}
        For any LCS $\Sigma_{G(\theta)}$ with nilrank zero satisfying the LARC, it holds:
        \begin{enumerate}
            \item If $\theta\neq\left(\begin{array}{cc}
    \gamma & -1 \\
     1 & \gamma
\end{array}\right)$ for all $\gamma\neq 0$, then $\Sigma_{G(\theta)}$ admits an infinite number of control sets with empty interior;

\item If $\theta=\left(\begin{array}{cc}
    \gamma & -1 \\
     1 & \gamma
\end{array}\right)$ for some $\gamma\neq 0$, then $\Sigma_{G(\theta)}$ is controllable.
        \end{enumerate}
        .
\end{theorem}

\begin{proof}
    By Proposition \ref{conjugation2}, we can assume, w.l.o.g., that 
    \begin{flalign*}
	  &&\left\{\begin{array}{l}
     \dot{t}=u\alpha\\
     \dot{v}=\Lambda_t^{\theta}\xi
\end{array}\right., u\in\Omega.  &&\hspace{-1cm}\left(\Sigma_{G(\theta)}\right)
\end{flalign*}

1. If $\theta\neq\left(\begin{array}{cc}
    \gamma & -1 \\
     1 & \gamma
\end{array}\right)$ for all $\gamma\neq 0$, there exists by Lemma \ref{positive} a vector $\hat{\xi}\in\R^2$ such that 
$\langle \Lambda_t^{\theta}\xi, \hat{\xi}\rangle\geq 0$ for all $t\in\R$. As a consequence, the smooth function
$$f:G(\theta)\rightarrow\R, \hspace{1cm}f(t, v):=\langle v, \hat{\xi}\rangle,$$
is such that the function $g_{{\bf u}}(s)=f(\phi_2(s, (t, v), {\bf u}))$, satisfies
$$g'_{{\bf u}}(s)=\frac{d}{ds}\langle \phi_2(s, (t, v), {\bf u}), \hat{\xi}\rangle=\langle \Lambda_{\phi_1(s, (t, v), {\bf u})}^{\theta}\xi, \hat{\xi}\rangle\geq 0,$$
and are non-decreasing. Therefore, $$(t_1, v_1), (t_2, v_2)\in G(\theta); \hspace{.5cm} \mbox{ with }\hspace{.5cm}f(t_1, v_1)<f(t_2, v_2),$$
cannot be in the same control set of $\Sigma_{G(\theta)}$ (see Remark \ref{function}). 

On the other hand, the fact that $(0, v)\in G(\theta)$ are fixed points of the drift $\XC=(0, \xi)$, implies that any such point is contained in a control set of $\Sigma_{G(\theta)}$. Therefore, for any $c\in\R$ the plane
$$P_c:=\{(t, v), \hspace{.3cm}\langle v, \hat{\xi}\rangle=c\},$$
contains (at least) one control of $\Sigma_{G(\theta)}$, showing the assertion.

2. Let $H=\R\cdot (0, \theta^{-1}\xi)$ and consider the homogeneous space $H\setminus G(\theta)$. By the calculations in Section 2.5.2, the canonical map is given by 
$$\pi: G(\theta)\rightarrow H\setminus G(\theta), \hspace{1cm} \pi(t, v)=(t, \langle v, R\theta^{-1}\xi\rangle).$$
Therefore, the induced system on $\R^2\simeq H\setminus G(\theta)$ is given by
 \begin{flalign*}
	  &&\left\{\begin{array}{l}
     \dot{t}=u\alpha\\
     \dot{x}=\langle \Lambda_t^{\theta}\xi, R\theta^{-1}\xi\rangle
\end{array}\right. . &&\hspace{-1cm}\left(\Sigma_{H\setminus G(\theta)}\right)
	  \end{flalign*}
Now, the fact that $\theta=\gamma I_{\R^2}+R$ implies that
$\rme^{t\theta}=\rme^{\gamma t}((\cos t) I_{\R^2}+(\sin t) R)$. As a consequence, 
$$\langle\Lambda_t^{\theta}\xi, R\theta^{-1}\xi\rangle=\langle(\rme^{t\theta}-I_{\R^2})\theta^{-1}\xi, R\theta^{-1}\xi\rangle=\langle\rme^{t\theta}\theta^{-1}\xi, R\theta^{-1}\xi\rangle$$
$$=\rme^{\gamma t}\cos t\langle\theta^{-1}\xi, R\theta^{-1}\xi\rangle+\rme^{\gamma t}\sin t\langle R\theta^{-1}\xi, R\theta^{-1}\xi\rangle=|\theta^{-1}\xi|^2\rme^{\gamma t}\sin t.$$

The rest of the proof is concluded in the next three steps.

{\bf Step 1:}  $\Sigma_{H\setminus G(\theta)}$ is controllable;

It is not hard to see that the solutions $\hat{\phi}$ of $\Sigma_{H\setminus G(\theta)}$ satisfy 
$$(i)\;\;\hat{\phi}_1(s, (t_0, x_0), u)=t_0+u\alpha s, \hspace{1cm}(ii)\;\;\hat{\phi}_2(s, (t_0, x_0), 0)=x_0+s|\theta^{-1}\xi|^2\rme^{\gamma t}\sin t, $$
$$\mbox{ and }\hspace{1cm}(iii)\;\;\hat{\phi}_2(s, (t_0, x_0), u)=\hat{\phi}_2(s, (t_0, 0), u)+x_0.$$

In particular, property $(i)$ implies that $$\forall t_1, t_2\in\R, \;\exists u\in\Omega; \hspace{1cm}\hat{\phi}(s, \{t_1\}\times\R, u)=\{t_2\}\times\R.$$

Therefore, it is always possible to construct a trajectory from any given point $(t, x)$ to the axis $\{0\}\times\R$ in positive and negative time. Hence, $\Sigma_{H\setminus G(\theta)}$ is controllable as soon as any two points in $\{0\}\times\R$ can be connected by a trajectory of the system.

Let us consider $x, y\in\R$ and assume w.l.o.g. that $x<y$. Fix $t_1, t_2\in\R$ satisfying 
$$-\pi/2<t_1<0<t_2<\pi/2,$$
and consider $s_{0, 1}, s_{2, 0}, s_{1, 2}\in (0, +\infty)$ and $u_{0, 1}, u_{2, 0}, u_{1, 2}\in\Omega$ such that 
 $$\hat{\phi}(s_{0, 1}, \{0\}\times\R, u_{0, 1})=\{t_1\}\times\R, \hspace{.5cm} \hat{\phi}(s_{2, 0}, \{t_2\}\times\R, u_{2, 0})=\{0\}\times\R\hspace{.5cm}\mbox{ and }\hspace{.5cm}\hat{\phi}(s_{1, 2}, \{t_1\}\times\R, u_{1, 2})=\{t_2\}\times\R.$$
Take $\hat{x}_i, \hat{y}_j$, $i, j=1, 2$ satisfying
$$\hat{\phi}(s_{0, 1}, (0, x), u_{0, 1})=(t_1, \hat{x}_1), \hspace{.5cm}\hat{\phi}(s_{0, 1}, (0, y), u_{0, 1})=(t_1, \hat{y}_1),$$
$$\hat{\phi}(s_{2, 0}, (t_2, \hat{x}_2), u_{2, 0})=(0, x)\hspace{.5cm}\mbox{ and }\hspace{.5cm}\hat{\phi}(s_{2, 0}, (t_2, \hat{y}_2), u_{2, 0})=(0, y).$$
Using property (iii), it is straightforward to show that there exist $z_1, z_2$ satisfying
$$z_1\leq\min\{\hat{x}_1, \hat{y}_1\}\hspace{.5cm}z_2\leq\min\{\hat{x}_2, \hat{y}_2\},\hspace{.5cm}\mbox{ and }\hspace{.5cm}\hat{\phi}_2(s_{1, 2}, (t_1, z_1), u_{1, 2})=z_2.$$
Furthermore, since $\sin t_1<0$ and $\sin t_2>0$ there exist $s_1, s_2, \hat{s}_1, \hat{s}_2\in [0, +\infty)$ such that 
$$\hat{x}_1+s_1|\theta^{-1}\xi|^2\rme^{\gamma t_1}\sin t_1=z_1, \hspace{1cm}\hat{y}_1+\hat{s}_1|\theta^{-1}\xi|^2\rme^{\gamma t_1}\sin t_1=z_1, $$
$$z_2+\hat{s}_2|\theta^{-1}\xi|^2\rme^{\gamma t_2}\sin t_2=\hat{x}_2\hspace{.5cm}\mbox{ and }\hspace{.5cm} z_2+s_2|\theta^{-1}\xi|^2\rme^{\gamma t_2}\sin t_2=\hat{y}_2.$$

Using the previous, a closed trajectory passing through $(0, x_1)$ to $(0, x_2)$ is obtained as follows:

\begin{enumerate}
    \item Go from $(0, x)$ to $(t_1, \hat{x}_1)$ in time $s_{0, 1}$ with the control $u_{0, 1}$;

    \item Go from $(t_1, \hat{x}_1)$ to $(t_1, z_1)$ in time $s_1$ with the control $0$;

    \item Go from $(t_1, z_1)$ to $(t_2, z_2)$ in time $s_{1, 2}$ with the control $u_{1, 2}$;

     \item Go from $(t_2, z_2)$ to $(t_2, \hat{y}_2)$ in time $s_2$ with the control $0$;

    \item Go from $(t_2, \hat{y}_2)$ to $(0, y)$ in time $s_{2, 0}$ with the control $u_{2, 0}$;

    \item Go from $(0, y)$ to $(t_1, \hat{y}_1)$ in time $s_{0, 1}$ with the control $u_{1, 0}$;

    \item Go from $(t_1, \hat{y}_1)$ to $(t_1, z_1)$ in time $\hat{s}_1$ with the control $0$;

    \item Go from $(t_1, z_1)$ to $(t_2, z_2)$ in time $s_{1, 2}$ with the control $u_{1, 2}$;

     \item Go from $(t_2, z_2)$ to $(t_2, \hat{x}_2)$ in time $\hat{s}_2$ with the control $0$;

    \item Go from $(t_2, \hat{x}_2)$ to $(0, x)$ in time $s_{2, 0}$ with the control $u_{2, 0}$;

\end{enumerate}

The previous implies that $\Sigma_{H\setminus G(\theta)}$
is controllable.

\bigskip

{\bf Step 2:} If $s_0(0, \theta^{-1}\xi)\in\inner\OC^{\pm}(e)$, there exists $s_1\in\R$ such that  
$$s_0s_1<0\hspace{.5cm}\mbox{ and }\hspace{.5cm} s_1(0, \theta^{-1}\xi)\in\inner\OC^{\pm}(e).$$

Let us show the claim only for the positive orbit, since the proof for the negative orbit is analogous.

By a straightforward calculation, we obtain that the second coordinate of the solutions of $\Sigma_{G(\theta)}$, for $u\neq 0$, is given by
$$\phi_2(s, (t_0, v_0), u)=\frac{1}{u\alpha}\rme^{t_0\theta}(\rme^{u\alpha s\theta}-u\alpha s\theta-I_{\R^2})\theta^{-2}\xi+s\Lambda_{t_0}^{\theta}\xi+v_0.$$
Let us choose $\rho\in\R$ such that $\pm\rho\alpha^{-1}\in\Omega$. Since $\phi_1(s, (t_0, v_0), u)=t_0+u\alpha s$, we get that
$$\phi(s, \phi(s, s_0(0, \theta^{-1}\xi), \rho\alpha^{-1}), -\rho\alpha^{-1})=(0, \phi_2(s, \phi(s, e, \rho\alpha^{-1}), -\rho\alpha^{-1})\in\{0\}\times\R^2.$$
Therefore, 
$$\phi_2(s, \phi(s, e, \rho\alpha^{-1}), -\rho\alpha^{-1})=\phi_2\left(s, \frac{1}{\rho}(\rme^{\rho s\theta}-\rho s\theta-I_{\R^2})\theta^{-2}\xi+s_0\theta^{-1}\xi, \rho\alpha^{-1}\right)$$
$$=-\frac{1}{\rho}\rme^{\rho s\theta}(\rme^{-\rho s\theta}+\rho s\theta-I_{\R^2})\theta^{-2}\xi+s\Lambda_{\rho s}^{\theta}\xi+\frac{1}{\rho}(\rme^{\rho s\theta}-\rho s\theta-I_{\R^2})\theta^{-2}\xi+s_0\theta^{-1}\xi$$
$$=-\frac{2}{\rho}\left(\theta^{-2}\xi+\rho s\theta^{-1}\xi-\rme^{\rho s\theta}\theta^{-2}\xi\right)+s_0\theta^{-1}\xi.$$
Using that $\rme^{t\theta}=\rme^{\gamma t}((\cos t) I_{\R^2}+(\sin t) R)$ and $(\gamma^2+1)\theta^{-1}=\gamma I_{\R^2}-R$, allows us to write such a solution on the orthogonal basis $\{\theta^{-1}\xi, R\theta^{-1}\xi\}$ as
$$\phi_2(s, \phi(s, e, \rho\alpha^{-1}), -\rho\alpha^{-1})=H_1(s)\cdot\theta^{-1}\xi+H_2(s)\cdot R\theta^{-1}\xi,$$
where
$$H_1(s)=\frac{2}{\rho(\gamma^2+1)}\left[\rme^{\gamma \rho s}(\gamma\cos(\rho s)+\sin(\rho s))-\gamma-s\rho(\gamma^2+1)\right]+s_0,$$
and 
$$H_2(s)=\frac{2}{\rho(\gamma^2+1)}\left[\rme^{\gamma \rho s}(\gamma\sin(\rho s)-\cos(\rho s))+1\right].$$

Also,
$$H_2(s)=0\hspace{.5cm}\iff\hspace{.5cm} \gamma\rme^{\gamma\rho s}\sin(\rho s)+1=\rme^{\gamma\rho s}\cos(\rho s),$$
implying that
$$H_1(s)=\frac{2}{\rho}\rme^{\gamma\rho s}\sin(\rho s)-2s+s_0, \hspace{.5cm}\mbox{ if }\hspace{.5cm} H_2(s)=0.$$
On the other hand, for any $k\in\Z$, we get that
$$H_2\left(\frac{2k\pi}{\rho}\right)=\frac{2}{\rho(\gamma^2+1)}(-\rme^{\gamma 2k\pi}+1)\hspace{.5cm}\mbox{ and }\hspace{.5cm} H_2\left(\frac{\pi+2k\pi}{\rho}\right)=\frac{2}{\rho(\gamma^2+1)}(\rme^{\gamma(\pi+2k\pi)}+1), $$
As a consequence, if $\gamma k>0$ and $\varepsilon>0$ is small enough, the function $H_2$ changes signs on the intervals $$I_k:=\left(\frac{2\pi k+\varepsilon}{\rho}, \frac{\pi+2\pi k-\varepsilon}{\rho}\right)\hspace{.5cm}\mbox{ and }\hspace{.5cm}\hat{I}_k:=\left(\frac{\pi+2\pi k+\varepsilon}{\rho}, \frac{2\pi (k+1)-\varepsilon}{\rho}\right).$$ 
Let us denote by $s_k\in I_k$ and $\hat{s}_k\in\hat{I}_k$ zeros of $H_2$ when $\gamma k>0$. For $|k|\rightarrow+\infty$, it holds:

\begin{enumerate}
    \item If $s_0>0$ and $\gamma>0$, we choose $\rho>0$ and $k\in\N$. In this case, 
    $$\gamma k>0, \hspace{.5cm}\hat{s}_k\rightarrow+\infty, \hspace{.5cm} H_2(\hat{s}_k)=0\hspace{.5cm}\mbox{ and }$$
    $$H_1(\hat{s}_k)=\frac{2}{\rho}\rme^{\gamma\rho \hat{s}_k}\sin(\rho \hat{s}_k)-2\hat{s}_k+s_0\rightarrow-\infty, \hspace{.5cm}\mbox{ since }\hspace{.5cm}\sin (\rho \hat{s}_k)\leq -\sin(\varepsilon)<0. $$

     \item If $s_0>0$ and $\gamma<0$, we choose $\rho<0$ and $-k\in\N$. In this case, 
    $$\gamma k>0, \hspace{.5cm}s_k\rightarrow+\infty, \hspace{.5cm} H_2(s_k)=0\hspace{.5cm}\mbox{ and }$$
    $$H_1(s_k)=\frac{2}{\rho}\rme^{\gamma\rho s_k}\sin(\rho s_k)-2s_k+s_0\rightarrow-\infty, \hspace{.5cm}\mbox{ since }\hspace{.5cm}\sin (\rho s_k)\geq\sin(\varepsilon)>0. $$

      \item If $s_0<0$ and $\gamma>0$, we choose $\rho>0$ and $k\in\N$. In this case, 
    $$\gamma k>0, \hspace{.5cm}s_k\rightarrow+\infty, \hspace{.5cm} H_2(s_k)=0\hspace{.5cm}\mbox{ and }$$
    $$H_1(s_k)=\frac{2}{\rho}\rme^{\gamma\rho s_k}\sin(\rho s_k)-2s_k+s_0\rightarrow+\infty, \hspace{.5cm}\mbox{ since }\hspace{.5cm}\sin (\rho s_k)\geq \sin(\varepsilon)>0. $$

      \item If $s_0<0$ and $\gamma<0$, we choose $\rho<0$ and $-k\in\N$. In this case, 
    $$\gamma k>0, \hspace{.5cm}\hat{s}_k\rightarrow+\infty, \hspace{.5cm} H_2(\hat{s}_k)=0\hspace{.5cm}\mbox{ and }$$
    $$H_1(\hat{s}_k)=\frac{2}{\rho}\rme^{\gamma\rho \hat{s}_k}\sin(\rho \hat{s}_k)-2\hat{s}_k+s_0\rightarrow+\infty, \hspace{.5cm}\mbox{ since }\hspace{.5cm}\sin (\rho \hat{s}_k)\leq -\sin(\varepsilon)<0. $$
\end{enumerate}

Therefore, in any case, there exists $S>0$ such that $H_2(S)=0$ and $s_0H_1(S)<0$. Hence, for $s_1=H_1(S)$, we get
$$s_1(0, \theta^{-1}\xi)=\phi(S, \phi(S, s_0(0, \theta^{-1}\xi), \rho\alpha^{-1}), -\rho\alpha^{-1})\in\phi(S, \phi(S, \inner\OC^+(e), \rho\alpha^{-1}), -\rho\alpha^{-1})\subset \inner\OC^+(e),$$
which proves the assertion.

\bigskip
{\bf Step 3:} If for some $s_0\in\R$ 
$$s_0(0, \theta^{-1}\xi)\in\inner\OC^{\pm}(e)\hspace{.5cm}\mbox{ then }\hspace{.5cm}\exists R>0; \;\forall s>R\hspace{.5cm}\implies \hspace{.5cm}\left\{\begin{array}{cc}
    s(0, \theta^{-1}\xi)\in\inner\OC^{\pm}(e) &  \mbox{ if } s_0>0 \\
   -s(0, \theta^{-1}\xi)\in\inner\OC^{\pm}(e)  & \mbox{ if } s_0<0 
\end{array}.\right.$$

Again, let us show the assertion only for the positive orbit. Since $\Sigma_{G(\theta)}$ satisfies the LARC, by \cite[Lemma 4.5.2]{CK} it holds that 
$$s_0(0, \theta^{-1}\xi)\in\inner\OC^{+}(e)\hspace{.5cm}\implies\hspace{.5cm}s_0(0, \theta^{-1}\xi)\in\inner_{\leq S}\OC^{+}(e)\stackrel{(\ref{properties}}{=}\inner_S\OC^{+}(e),$$
for some $S>0$. Consequently, there exists an interval $(a, b)\subset\R$ satisfying $\frac{s_0}{|s_0|}\cdot (a, b)\subset(0, +\infty)$ and 
$$\forall s\in(a, b), \hspace{1cm} s(0, \theta^{-1}\xi)\in \inner\OC_S^+(e).$$
On the other hand, the fact that points in $\{0\}\times\R^2$ are fixed points of the drift allows us to obtain that
    $$(s_1+s_2)(0, \theta^{-1}\xi)=s_1(0, \theta^{-1}\xi)s_2(0, \theta^{-1}\xi)=\varphi_S(s_1(0, \theta^{-1}\xi))s_2(0, \theta^{-1}\xi)\in\inner_{2S}\OC^+(e)\subset\inner\OC^+(e),$$
    and, inductively, we conclude that
    $$\forall n\in\N, s\in (na, nb), \hspace{1cm} s(0, \theta^{-1}\xi)\in\inner\OC^+(e).$$
    Since, $a, b$ have the same sign as $s_0$, there exists $R>0$ such that
    $$\frac{s_0}{|s_0|}(R, +\infty)\subset \bigcup_{n\in\N}(na, nb),$$
    and the assertion follows.

\bigskip

{\bf Step 4:} $\Sigma_{G(\theta)}$ is controllable.

Let us consider $(t, v)\in\inner\OC^+(e)$, which exists by the LARC. By Step 1., there exists $s>0$ and ${\bf u}\in\UC$ such that 
$$\pi(\phi(s, (t, v), {\bf u}))=\hat{\phi}(s, \pi(t, v), {\bf u})=(0, 0)=\pi(e)\hspace{.5cm}\implies\hspace{.5cm} \phi(s, (t, v), {\bf u})\in H.$$
Consequently, there exists $s_0\in\R$ such that 
$$s_0(0, \theta^{-1}\xi)=\phi(s, (t, v), {\bf u})\subset \phi(s, \inner\OC^+(e), {\bf u})\subset\inner\OC^+(e).$$
By Step 2., the previous implies the existence of $s_1\in\R$ with $s_0s_1<0$ and such that $s_1(0, \theta^{-1}\xi)\in\inner\OC^{+}(e)$.

Now, Step 3. applied for $s_0$ and $s_1$ assures the existence of $S>0$ great enough, such that 
$$S(0, \theta^{-1}\xi)\in\inner\OC^+(e)\hspace{.5cm}\mbox{ and }\hspace{.5cm} -S(0, \theta^{-1}\xi)\in \inner\OC^+(e).$$
Since, by \cite[Lemma 4.5.2]{CK}, there exists $S_1, S_2>0$ such that
$$S(0, \theta^{-1}\xi)\in\inner_{S_1}\OC^+(e)\hspace{.5cm}\mbox{ and }\hspace{.5cm} -S(0, \theta^{-1}\xi)\in\inner_{S_2}\OC^+(e),$$
we conclude that 
$$e=(S(0, \theta^{-1}\xi))(-S(0, \theta^{-1}\xi))=\varphi_{S_2}(S(0, \theta^{-1}\xi))(-S(0, \theta^{-1}\xi))\in \varphi_{S_2}(\inner\OC_{S_1}^+(e))\inner\OC_{S_2}^+(e)$$
$$\stackrel{\ref{properties}}{=}\inner\OC_{S_1+S_2}^+(e)\subset\inner\OC^+(e),$$
which, by Theorem \ref{subgrupos} implies the controllability of $\Sigma_{G(\theta)}$, concluding the proof.
\end{proof}

\begin{remark}
    The previous result is quite remarkable, since it shows how strong the influence of the group structure is on the controllability of linear control systems. Moreover, it shows that LARC is not always enough to assure the existence of a control set with a nonempty interior.
\end{remark}

\section{Connected non-simply connected groups}

In this section, we analyze the LCSs on 3D solvable Lie groups $G$ that are not simply connected. The analysis is done by using the lift of LCSs from $G$ to its simply connected cover, as commented in Section 2.4.

Due to the characterization present in \cite[Chapter 7]{onis}, the only connected, solvable, nonnilpotent 3D Lie groups associated with the Lie algebras $\fg(\theta)$ that are also non-simply connected appear when 
$$\theta=\left(\begin{array}{cc}
   0  & -1 \\
    1 & 0
\end{array}\right)\hspace{.5cm}\mbox{ or }\hspace{.5cm}\theta=\left(\begin{array}{cc}
   1  & 0 \\
    0 & 0
\end{array}\right).$$

In the next sections, we consider, separately, both cases.

\subsection{The group of rigid motions and its $n$-fold covers}

If $\theta=\left(\begin{array}{cc}
   0  & -1 \\
    1 & 0
\end{array}\right)$ then, for each $n\in\N$
$$Z_n:=\{(2kn\pi, 0)\in G(\theta), k\in\Z\}$$ 
is a discrete central subgroup of $G(\theta)$. In particular, the quotients $Z_n\setminus G(\theta)$ are connected Lie groups associated with the Lie algebra $\fg(\theta)$. The group $SE(2):=Z_1\setminus G(\theta)$ is the group of {\it proper motions} of $\R^2$. It is the connected component of the group of rigid motions of $\R^2$. For $n\geq 2$ the group $SE(2)_n:=Z_n\setminus G(\theta)$ is a $n$-fold cover of $SE(2)$. 

The canonical projection $\pi_n: G(\theta)\rightarrow Z_n\setminus G(\theta)$ is given by 
$$\pi(t, v)=([t], v), \hspace{.5cm} [t]:=t+2n\pi\Z.$$

By the results in Sections 2.3, in order to analyze the behavior of LCSs on $\Sigma_{SE(2)_n}$, it is enough to understand LCSs on $G(\theta)$ whose flow of the drift let the subgroups $Z_n$ invariant. However, using the expression (\ref{linearvector}), we have that 
$$\varphi_s(2nk\pi, 0)=(2nk\pi,\rme^{sA}0+ \Lambda^{\theta}_{2nk\pi}\Lambda_s^A\xi)=(2nk\pi, 0), \hspace{.5cm}\mbox{ where we used that }\hspace{.5cm}\Lambda^{\theta}_{2nk\pi}\equiv 0,$$
implying that a linear vector field on $SE(2)_n$ is given by 
$$\XC([t], v)=(0, Av+\Lambda_t^{\theta}\xi).$$
On the other hand, if $\det A\neq 0$, a LCS on $G(\theta)$ is conjugated to the control-affine system $\Sigma_{\R\times\R^2}$ through the diffeomorphism
$$\psi(t, v)=(t, \rho_{-t}(v+\Lambda_t^{\theta}A^{-1}\xi)).$$
By Proposition \ref{controlsetsigma1}, for all control $u$ in the connected component of the zero in $\widehat{\Omega}$, the fiber 
$\R\times \{v(u)\}$
is contained in the interior of the control set $\CC_{\R\times\R^2}$ of $\Sigma_{\R\times\R^2}$. As a consequence, 
$$\psi^{-1}(\R\times \{v(u)\})=\{(t, \rho_{t}(v(u)-\Lambda_tA^{-1}\xi)), \;\;t\in\R\}\subset\inner\CC_{G(\theta)}.$$

We can now prove the main result for the control set of LCSs on $SE(2)_n$.

\begin{theorem}
Let $\Sigma_{SE(2)_n}$ be a LCS on $SE(2)_n$ satisfying the LARC. It holds:
\begin{itemize}
    \item[(i)] If $\det A\neq 0$ then $\Sigma_{SE(2)_n}$ admits a unique control set $\CC_{SE(2)_n}$
satisfying 
$$\pi_n^{-1}\left(\CC_{SE(2)_n}\right)=\CC_{G(\theta)},$$
where $\pi_n:G(\theta)\rightarrow SE(2)_n$ is the canonical projection and $\CC_{G(\theta)}$ is the unique control set of the LCS on $G(\theta)$ that is $\pi_n$-conjugated to $\Sigma_{SE(2)_n}$.

\item[(ii)] If $\det A= 0$ then $\Sigma_{SE(2)_n}$ admits an infinite number of control sets with empty interior.
\end{itemize}
\end{theorem}

\begin{proof}
    (i) By Proposition \ref{conjugation} there exists a control set $\CC_{SE(2)_n}$ of the system $\Sigma_{SE(2)_n}$ satisfying $\pi_n\left(\CC_{G(\theta)}\right)\subset \CC_{SE(2)_n}$ and the equality holds if we show that $$\pi_n^{-1}(\pi_n(t, v))\subset \inner\CC_{G(\theta)},\hspace{.5cm}\mbox{ 
    for some }\hspace{.5cm}(t, v)\in\inner \CC_{G(\theta)}.$$
    Since for all $k,n\in\N$ we have that $\rho_{2nk\pi}=I_{\R^2}$, we get that 
    
    $$\forall u\in\widehat{\Omega}, \hspace{.5cm}\pi_n^{-1}(\pi_n(0, v(u)))=\{(2nk\pi, v(u)), k\in\Z\}$$
    $$=\left\{\left(2nk\pi, \rho_{2nk\pi}(v(u))-\Lambda_{2nk\pi}^{\theta}A^{-1}\xi)\right), k\in\Z\right\}\subset\psi^{-1}(\R\times\{v(u)\})\subset\inner\CC_{G(\theta)},$$
    implying that $\pi_n^{-1}(\CC_{SE(2)_n})=\CC_{G(\theta)}$ and showing the assertion.

        (ii) Since $A$ and $\theta$ commute, $\det A=0$ if and only if $A\equiv 0$. Therefore, up to isomorphisms, $\Sigma_{SE(2)_n}$ is given,in coordinates, by 
\begin{flalign*}
	  &&\left\{\begin{array}{l}
     \dot{[t]}=u\alpha\\
     \dot{v}=\Lambda_t^{\theta}\xi
\end{array}\right., u\in\Omega.  &&\hspace{-1cm}\left(\Sigma_{SE(2)_n}\right)
\end{flalign*}
        
        Let $\hat{\xi}$ as given in Lemma \ref{positive} and define the function 
        $$f:SE(2)_n\rightarrow \R, \hspace{1cm} f([t], v):=\langle v, \hat{\xi}\rangle.$$
       Analogously as in the proof of Theorem \ref{kerA=2}, one can show that the functions $g_{{\bf u}}(s)=\langle \phi(s, ([t], v), {\bf u}), \hat{\xi}\rangle$ satisfy
    $$g'_{{\bf u}}(s)=\frac{d}{ds}\langle \phi_2(s, ([t], v), {\bf u}), \hat{\xi}\rangle=\langle \Lambda_{\phi_1(s, ([t], v), {\bf u})}^{\theta}\xi, \hat{\xi}\rangle\geq 0,$$
and are nondecreasing. Therefore, if  
$$([t_1], v_1), ([t_2], v_2)\in G(\theta); \hspace{.5cm} \mbox{ satisfy 
 }\hspace{.5cm}f([t_1], v_1)<f([t_2], v_2),$$
they cannot be in the same control set of $\Sigma_{SE(2)_n}$ (see Remark \ref{function}). 

On the other hand, the points $([0], v)\in SE(2)_n$ are fixed points of the drift $\XC=(0, \xi)$, and hence, any such point is contained in a control set of $\Sigma_{SE(2)_n}$. Therefore, for any $r\in\R$ the cylinder
$$C_r:=\{([t], v), \hspace{.3cm}\langle v, \hat{\xi}\rangle=r\},$$
contains (at least) one control of $\Sigma_{SE(2)_n}$, concluding the proof.
\end{proof}

\begin{remark}
    A much more detailed analisys of the control sets of $SE(2)$ was done in \cite{DSAyAOR}.
\end{remark}

\subsection{The group $\mathrm{Aff}(2)\times S^1$}

If $\theta=\left(\begin{array}{cc}
   1  & 0 \\
    0 & 0
\end{array}\right)$ the group $G(\theta)$ can be seen as the Cartesian product $\mathrm{Aff}(\R)\times\R$. In fact, since $v\in\R^2$ is written uniquely as $v=x\mathbf{e}_1+y\mathbf{e}_2$, the map
\begin{equation}
    \label{difeo}
    (t, v)\in G(\theta)\mapsto ((t, x), y)\in\mathrm{Aff}(\R)\times \R,
\end{equation}
is a diffeomorphim. Moreover, 
$$\rho_tv=x\rho_t\mathbf{e}_1+y\rho_t\mathbf{e}_2=x\rme^t\mathbf{e}_1+y\mathbf{e}_2,$$
implies that 
$$(t_1, v_1)(t_2, v_2)=(t_1+ t_2, v_1+\rho_{t_1}v_2)= (t_1+t_2, (x_1+\rme^{t_1}x_2)\mathbf{e}_1+(y_1+y_2)\mathbf{e}_2)$$
$$\mapsto((t_1+t_2, x_1+\rme^{t_1}x_2), y_1+y_2)=((t_1, x_1)(t_2, x_2), y_1+y_2)=((t_1, x_1), y_1)((t_2, x_2), y_2),$$
which shows that the map (\ref{difeo}) is an isomorphism.

Following \cite[Chapter 7]{onis}, up to isomorphisms, the only discrete central subgroup of $G(\theta)=\mathrm{Aff}(\R)\times\R$ is $Z:=\{((0, 0), 2k\pi), k\in\Z\}$. Therefore, $G(\theta)/Z=\mathrm{Aff}(\R)\times S^1$ is the unique connected, non-simply connected Lie group with the Lie algebra $\fg(\theta)=\mathfrak{aff}(\R)\times\R$, where $S^1=\R/\Z$. The canonical projection is given by  
$$\pi: \mathrm{Aff}(\R)\times\R\rightarrow \mathrm{Aff}(\R)\times S^1, \hspace{.5cm} ((t, x), y)\mapsto ((t, x), [y]).$$ 

By Section 2.3, the linear vector fields of $\mathrm{Aff}(\R)\times S^1$ are completely determined by the linear vector fields on $G(\theta)=\mathrm{Aff}(\R)\times\R$ whose flows let the subgroup $Z$ invariant. Let then $\XC$ be a linear vector field on $G(\theta)$ with associated flow $\{\varphi_s\}_{s\in\R}$ and assume that 
$$\varphi_s(0,2k\pi e_2)\in Z, \hspace{1cm}\forall s\in\R, k\in\Z.$$
Using the expression in (\ref{linearvector}) for $\varphi_s$, one obtains that,
$$\varphi_s(0,2k\pi e_2)=(0, 2k\pi \mathrm{e}^{sA}e_2)\in Z\hspace{.5cm}\iff\hspace{.5cm} \rme^{sA}e_2\in 2\pi\Z e_2\hspace{.5cm}\iff\hspace{.5cm} Ae_2=0.$$
%In particular, the associated matrix satisfies $\det A=0$.

\begin{theorem}
    Any linear control system $\Sigma_{\mathrm{Aff}(\R)\times S^1}$ on $\mathrm{Aff}(\R)\times S^1$ satisfying the LARC, admits a unique control set $\CC_{\mathrm{Aff}(\R)\times S^1}$
satisfying:
  $$\CC_{\mathrm{Aff}(\R)\times S^1}=\pi_1^{-1}(\CC_{\mathrm{Aff}(2)})=\CC_{\mathrm{Aff}(\R)}\times S^1,$$
where $\CC_{\mathrm{Aff}(\R)}$ is the unique control set of the LCS on $\mathrm{Aff}(\R)$ that is $\pi_1$-conjugated to $\Sigma_{\mathrm{Aff}(\R)\times S^1}$, where $\pi_1:\mathrm{Aff}(\R)\times S^1\rightarrow \mathrm{Aff}(\R)$ is the canonical projection. Moreover, $\CC_{\mathrm{Aff}(\R)\times S^1}$ is open if $\tr A>0$, closed if $\tr A<0$, and equal to $\mathrm{Aff}(\R)\times S^1$ if $\tr A=0$.

\end{theorem}

\begin{proof} Let us assume, w.l.o.g., that the left-invariant vector field of $\Sigma_{\mathrm{Aff}(\R)\times S^1}$ is $Y=(\alpha, 0)$, with $\alpha\neq 0$. In this case, our system in $\mathrm{Aff}(\R)\times S^1$, is given by 
 \begin{flalign*}
	  &&\left\{\begin{array}{l}
     \dot{t}=u\alpha\\
     \dot{x}=\lambda x+(\rme^{t}-1)\xi_1\\
     \dot{[y]}=t\xi_2
\end{array}\right., u\in\Omega,  &&\hspace{-1cm}\left(\Sigma_{\mathrm{Aff}(\R)\times S^1}\right)
\end{flalign*}
where $\lambda=\tr A$. Moreover, the LARC is equivalent to $\alpha\xi_1\xi_2\neq 0$. 

If $\lambda\neq 0$, the system satisfy the ad-rank condition, and hence, there exists a unique control set $\CC_{\mathrm{Aff}(\R)\times S^1}$ with a nonempty interior. Moreover, since the points in $\{(0, 0)\}\times S^1$ are fixed by the flow of $\XC=(A, \xi)$, we get that $\{(0, 0)\}\times S^1\subset\inner \CC_{\mathrm{Aff}(\R)\times S^1}$ (see Theorem \ref{subgrupos}). On the other hand, the projection 
$$\pi: \mathrm{Aff}(\R)\times S^1\rightarrow \mathrm{Aff}(\R), \hspace{1cm}\pi_1((t, x), y)=(t, x),$$
conjugates $\Sigma_{\mathrm{Aff}(\R)\times S^1}$ and the linear control system 
\begin{flalign*}
	  &&\left\{\begin{array}{l}
     \dot{t}=u\alpha\\
     \dot{x}=\lambda x+ (\rme^{t}-1)\xi_1\\
\end{array}\right., u\in\Omega.  &&\hspace{-1cm}\left(\Sigma_{\mathrm{Aff}(\R)}\right)
\end{flalign*}
Since $(0, 0)\in\inner\CC_{\mathrm{Aff}(\R)}$ and $\pi_1^{-1}(0, 0)=\{(0, 0)\}\times S^1\subset \inner\CC_{\mathrm{Aff}(\R)\times S^1}$ we obtain, by Proposition \ref{conjugation}, that
$$\CC_{\mathrm{Aff}(\R)\times S^1}=\pi_1^{-1}(\CC_{\mathrm{Aff}(2)})=\CC_{\mathrm{Aff}(\R)}\times S^1.$$
In particular, $\CC_{\mathrm{Aff}(\R)\times S^1}$ is open when $\lambda>0$ and closed when $\lambda<0$ since the same holds for $\CC_{\mathrm{Aff}(\R)}$ (see \cite[Theorem 3.6]{DSAy0}).

Let us assume now that $\lambda= 0$. In this case, the canonical projection 
$$\pi_1: \mathrm{Aff}(\R)\times S^1\rightarrow \mathrm{Aff}(\R), \hspace{1cm}((t, x), y)\mapsto (t, x),$$
conjugates $\Sigma_{\mathrm{Aff}(\R)\times S^1}$ to the LCS 
\begin{flalign*}
	  &&\left\{\begin{array}{l}
     \dot{t}=u\alpha\\
     \dot{x}=(\rme^{t}-1)\xi_1\\
\end{array}\right., u\in\Omega,  &&\hspace{-1cm}\left(\Sigma_{\mathrm{Aff}(\R)}\right)
\end{flalign*}
on $\mathrm{Aff}(\R)$. By \cite[Theorem 3.2]{DSAy0}), we have that $\Sigma_{\mathrm{Aff}(\R)}$ is controllable. 

Let us consider $P\in\inner\OC^+(e)$, whose existence is assured by the LARC. The controllability of the projected system $\Sigma_{\mathrm{Aff}(\R)}$, implies the existence of ${\bf u}\in\UC$ and $s>0$ such that
$$\pi_1(\phi(\tau, P, {\bf u}))=\phi_1(\tau, \pi_1(P), {\bf u})=(0, 0),$$
and hence
$$((0, 0), [y_0])=\phi(s, P, {\bf u})\in\phi\left(s, 
\inner\OC^+(e), {\bf u} \right)\subset \inner\OC^+(e).$$

Moreover, the LARC together with \cite[Lemma 4.5.2]{CK} and Proposition \ref{properties}, imply the existence of $S>0$ such that $((0, 0), [y_0])\in \inner\OC_S^+(e)$.  Since $\inner\OC_S^+(e)$ is open, there exists $r\in\mathbb{Q}$ with $((0, 0), [r])\in \inner\OC_S^+(e)$. On the other hand, the fact that points in $\{(0, 0)\}\times S^1$ are fixed points of the drift allows us to obtain that
    $$((0, 0), [2r])=((0, 0), [r])((0, 0), [r])=\varphi_S((0, 0), [r])((0, 0), [r])\in\inner_{2S}\OC^+(e),$$
    and, inductively, that 
    $$((0, 0), [mr])\in\inner_{mS}\OC^+(e), \hspace{1cm}\forall m\in\N.$$
    Since $r\in\mathbb{Q}$, there exists $m_0\in\N$ such that $m_0 r\in\Z$ implying that 
    $$e=((0, 0), [0])=((0, 0), [m_0 r])\in \inner\OC_{m_0S}^+(e)\subset\inner\OC^+(e).$$ 
    Moreover, the fact that $\tr A=0$ implies $A\equiv 0$, gives us that the associated derivation $\DC$ has only eigenvalues equal to zero. Consequently, Theorem \ref{subgrupos} implies that $\Sigma_{\mathrm{Aff}(\R)\times S^1}$ is controllable, concluding the proof.
\end{proof}

\begin{remark}
Let us note the big difference between $\pi_1$-conjugated LCSs on $\mathrm{Aff}(\R)\times S^1$ and on its simply connected covering $\mathrm{Aff}(\R)\times \R$ when the associated matrix $A\equiv 0$. Precisely, $\Sigma_{\mathrm{Aff}(\R)\times \R}$ admits an infinite number of control sets with empty interiors, while its projection $\Sigma_{\mathrm{Aff}(\R)\times S^1}$ is controllable.
\end{remark}

\end{document}